\documentclass[12pt]{amsart}
 
\usepackage{graphicx}
\usepackage[utf8]{inputenc}
\usepackage[english]{babel}
\usepackage{amssymb}
\usepackage{csquotes}
   
\usepackage{amsmath,amsfonts,pictex,graphicx,fullpage,etex,hyperref}

\usepackage{todonotes} 
\usepackage{diagrams}
\newarrow{DashTo}{}{dash}{}{dash}{>}

\usepackage{enumerate}
\usepackage{enumitem}
\newtheorem{theorem}{Theorem}[section]

\newtheorem{lemma}[theorem]{Lemma}
\newtheorem{proposition}[theorem]{Proposition}
\newtheorem{corollary}[theorem]{Corollary}

\newtheorem{definition}[theorem]{Definition}

\theoremstyle{remark}

\numberwithin{equation}{section}

\newcommand{\Z}{\mathbb{Z}}

\newcommand{\Q}{\mathbb{Q}}
\newcommand{\R}{\mathbb{R}}

\begin{document}

\date{\today}

\title[Smooth factors]{Smooth Factors of Projective Actions of Higher Rank Lattices and Rigidity}
\author[A. Gorodnik]{Alexander Gorodnik $^\ast$}

\author[R. Spatzier]{Ralf Spatzier $^{\ast\ast}$}

\thanks{$^\ast$ Supported in part by EPSRC grant EP/H000091/1,  ERC grant 239606 and MSRI}

\thanks{$^{\ast \ast}$ Supported in part by NSF grant DMS--1307164 and MSRI}

\address{Department of Mathematics, University of Bristol, Bristol, BS8 1TW, U.K. }

\email{a.gorodnik@bristol.ac.uk}

 \address{Department of Mathematics, University of Michigan, Ann Arbor, MI 48109.}

\email{spatzier@umich.edu}

\begin{abstract} We   study smooth factors of the standard actions of  lattices in higher rank semisimple Lie groups on flag manifolds. Under a mild condition on existence of a single differentiable sink, we show that these factors are $C^{\infty}$-conjugate to the standard actions on flag manifolds. 
\end{abstract}

\maketitle

\section{Introduction}

Let $\Gamma$ be a lattice in a connected semisimple Lie group $G$, and let $P$ be  a parabolic subgroup of $G$.  In this paper, we will be interested in the action of $\Gamma$ on the flag manifold $F=G/P$ by left translations.
% preserving the measure class of Haar measure. 
The simplest example is given by the linear action of a lattice in
$\hbox{SL}_n(\mathbb{R})$ on the projective space $\mathbb{P}^{n-1}$. 
These actions have played a major role in the Rigidity Theory. In particular, understanding their dynamics proved crucial in Margulis' Superrigidity and Finiteness Theorems.
Margulis \cite{margulis} classified all measurable factors of the $\Gamma$-action on $F=G/P$ 
when $G$ has real rank at least two, and the lattice $\Gamma$ is irreducible.
He showed that every such factor is  measurably isomorphic to the $\Gamma$-action on $F'=G/Q$ where $Q$ is a parabolic subgroup containing $P$.
This was one of the ingredients in his proof of Margulis'
 Finiteness Theorem  which shows that all normal subgroups of $\Gamma$ are either of finite index or central.  %(see \cite{margulis-book} and \cite{zimmer}). 
Dani \cite{Dani} analysed topological factors of these actions when $G$ has no compact or real rank-one factors.  He proved  that any Hausdorff factor
of the action of $\Gamma$ on $F=G/P$ is $C^0$-conjugate to the action of $\Gamma$ on $F'=G/Q$, where $Q$ is a parabolic subgroup containing $P$.
The aim of this paper is to establish a smooth analogue of Margulis' and Dani's factor theorems.  
 Our results also complement recent work by Brown, Rodriguez Hertz and Wang on low dimensional actions of higher rank lattices, and provide partial solutions and further evidence  for their Conjecture 1.8  \cite{BRHW2016}.

The actions on flag manifolds are very different from measure-preserving actions, and
one of their essential features is existence of sinks.
Since analysis of dynamics in neighbourhoods of sinks will play central role in our
discussion, we give precise definitions.

\begin{definition} \label{d:sink}
{\rm
Let $p$ be a fixed point of a $C^1$-map $f$ on a manifold $M$.
\begin{enumerate}
	\item[(i)] The point $p$ is called a \emph{topological sink} if there exists a neighbourhood $W_0$ of $p$ such that for every neighbourhood $W$ of $p$ and all sufficiently large $n$, we have $f^n(W_0)\subset W$.
	\item[(ii)] The point $p$ is called a
	\emph{differentiable sink} if all eigenvalues of $D(f)_p$ have modulus less than one,
	or equivalently, there exists a Riemannian metric in a neighbourhood $U$ of $p$ such that $\|D(f)_x\|<1$ for all $x\in U$.
\end{enumerate}
}
\end{definition}

It is easy to see that a differentiable sink is always a topological sink, while the converse fails. In the case of actions on flag manifolds $F=G/P$, every $\mathbb{R}$-regular
element $g\in G$ has a unique differentiable sink. We recall that an element $g\in G$ is called \emph{$\mathbb{R}$-regular} if the number of its eigenvalues, counted with multiplicity, of $\hbox{Ad}(g)$ is minimal possible.

The main result of the paper is

\begin{theorem}\label{th:main}
Let $G$ be a connected semisimple Lie group without compact or real rank-one factors. Let $P$ be a parabolic subgroup of $G$ and $F=G/P$ the corresponding flag manifold.
Let $\Gamma$ be a lattice in $G$.
We denote by $\rho_0$ the standard action of $\Gamma$ on $F$.
Let $\rho $ be a $C^{\infty}$-action of $\Gamma$ on a  manifold $M$
such that for some $\gamma \in \Gamma$, the transformation $\rho(\gamma)$ has a differentiable sink in $M$. 
Suppose   $\psi: F \to M$ is a $C^0$-semi-conjugacy between $\rho _0$ and $\rho$.  Then there exist a parabolic subgroup $Q$ containing $P$ and a $C^{\infty}$-smooth $\Gamma$-equivariant diffeomorphism $\phi: G/Q \to M$ such that 
$ \phi = \psi \circ \pi$, where $\pi: G/P \to G/Q$ is the canonical factor map.
\begin{diagram}
	F= G/P & & \\
	\dTo^\psi  & \rdTo^\pi &\\ 
	M & \lDashTo_\phi & G/Q
\end{diagram}
\end{theorem}

It is clear that any $C^0$-factor of a standard action on a flag manifold has topological sinks as the original action does.  However, we don't know whether the existence of a differentiable sink is automatic for $C^0$-factors of standard actions.  There are examples of smooth lattice actions on $S^1$-bundles over flag manifolds which have topological sinks that are not differentiable sinks.
(see Section \ref{sec:sink} below).  

The assumption regarding existence of a differentiable sink in Theorem \ref{th:main} is 
similar to the hyperbolicity assumptions that appeared in previous works on rigidity
in the setting of measure-preserving actions.  Our main result could be considered as an analogue of the smooth rigidity theorems for higher rank Anosov actions \cite{Katok-Spatzier,Fisher-Kalinin-Spatzier,RH-Wang}. There again one proves regularity of a $C^0$-conjugacy under suitable uniform hyperbolicity hypothesises.

As an application of our main result, we also get local rigidity results. Given a $C^{\infty}$-action $\rho _0$ of a finitely generated group $\Gamma$ on a compact manifold $M$, we call
the action $\rho _0$  $C^1$-{\em locally rigid} if any other $C^{\infty}$-action $\rho $ of $\Gamma$ on $M$ is $C^{\infty}$-conjugate to $\rho _0$ provided that for a finite set of generators $S$ of $\Gamma$, the maps $\rho (s)$ and $\rho _0 (s)$ are sufficiently close in the $C^1$-topology for all $s \in S$.  There is by now a long history of local rigidity results for higher rank actions and, in particular, higher rank lattices (see \cite{Ghys, Kanai, Katok-Spatzier,  Fisher-Margulis} amongst others and \cite{Fisher} for a survey).  The following result was already obtained by M. Kanai in \cite{Kanai} under a more stringent closeness condition and by A. Katok and R. Spatzier in \cite{Katok-Spatzier}.

\begin{corollary}  \label{local rigidity}
Let $G$ and $P$ be as above.  Let $\Gamma$ be a unifom lattice in $G$.  Then the $\Gamma$-action on $F=G/P$ is locally $C^1$-rigid.  
\end{corollary}

 We obtain this as an almost immediate corollary of our main theorem. In fact, following Ghys \cite{Ghys},  if $\Gamma$ is cocompact, Katok and Spatzier \cite{Katok-Spatzier} constructed a $C^1$-close perturbation of the action of a Cartan subgroup $A$ on $G/\Gamma$ by left translations that corresponds to the perturbation of the $\Gamma$-action on $G/P$.  This action is Anosov, and thus structurally stable.  This in turn gives a $C^0$-conjugacy between the action of $\Gamma$ on the flag manifold $F$ and its perturbation.  Hence, now Corollary \ref{local rigidity} follows immediately from Theorem \ref{th:main}.  In a sense our argument is dual to the argument in \cite{Katok-Spatzier}. 
 
 We do not know if local rigidity holds for actions of non-uniform lattices on flag manifolds.  Our argument shows that it suffices to prove structural stability for such actions. 

Our results fail  for cocompact lattices $\Gamma$ in $\hbox{PSL}_2(\R)$. Indeed, the moduli space of discrete faithful cocompact representations of $\Gamma$ in $\hbox{PSL}_2(\R)$ is non-trivial (in fact, it has positive dimension by standard results of Teichm\"{u}ller theory). Let us take two representations $\rho_1$ and $\rho _2$ that are not equivalent, and consider the resulting extension of the  natural actions by isometries on the hyperbolic  plane ${\mathcal H}^2$.  The  boundary circle $\partial {\mathcal H}^2$ of ${\mathcal H}^2$ is naturally identified with one-dimensional projective space $\R P ^1 \simeq \hbox{PSL}_2(\R) / P$, where $P$ denotes the parabolic subgroup of $\hbox{PSL}_2(\R)$ consisting of upper triangular matrices.  Both  $\rho _1$ and $\rho _2$ define $C^{\infty}$-actions of $\Gamma$ on $\partial {\mathcal H}^2$.  Since $\rho _1 (\Gamma)$ and $\rho _2 (\Gamma)$  are quasi-isometric, their actions on the boundary are $C^0$-conjugate.   The conjugating homeomorphism however  cannot be differentiable 
even at a single point by well-known results of Ivanov \cite{Ivanov} and Tukia \cite{Tukia}. Similar constructions can be made for Zariski dense groups in higher dimension, for example, for convex cocompact groups via quasi-conformal deformations in 3-dimensional hyperbolic space.
Typically, one has a large moduli space of representations for Zariski dense groups, and we expect that the above  considerations generalize to give counterexamples
for actions of Zariski dense subgroups on flag manifolds.
%Moreover one can deform   Zariski dense free subgroups in any semi simple Lie group. Their induced projective actions will not be $C^1$-conjugate as we may change the derivatives at sinks.  Other counterexamples come from  Benoist' projective groups which have large deformation space.   
It is not clear to us whether lattices in other real rank-one groups have more rigidity.  We remark that both Margulis' and Dani's theorem  fail in the real rank-one case (see \cite{margulis, spatzier, zimmer0}).

%In higher rank, there may be counterexamples coming form Benoist' projective groups.  CHECK. 
%We don't know if it suffices to assume minimality.  

%However, if the map is Holder then everything maybe still works. (??????)
%{\color{red}corollary about Hoelder maps?}

\subsection*{Organisation of the paper}
In the next section we recall basic properties of the standard actions on flag manifolds $F$.
Next, in Section \ref{sec:many sinks} we investigate general smooth actions on manifold $M$
which are continuously conjugate to the standard actions and establish existence of many sinks in $M$.
In Section \ref{sec:projections}, we analyse properties of the conjugacy map $F\to M$
further and introduce projection maps 
to certain dynamically defined submanifolds of a sink.
These maps are defined on open subsets of $F$ and $M$ and are intertwined by the conjugacy $F\to M$. Then in Section \ref{s:foliations} we establish smoothness of the map $F\to M$ along a family of foliations. Finally, in Section \ref{sec:final}
we complete the proof combining results from Sections \ref{sec:projections} and \ref{s:foliations}. In Section \ref{sec:sink}, we give an example of a lattice action with 
a topological sink which is not a differentiable sink.

\subsection*{Acknowledgements}
We are grateful to  David Fisher and Gopal Prasad for many pertinent discussions related to this paper.   In addition we thank Tasho Kaletha and Loren Spice for discussions on Lie theory.    A.G. would like to thank the University of Michigan for hospitality
during  visits while  work on this project proceeded.
R.S. thanks the University of Bristol for hospitality and support during this work.   Both authors also  
benefitted from the   special semester on Dynamics on Moduli Spaces of Geometric Structures at MSRI, and thank the Institute for its support and hospitality. 

%%%%%%%%%%%%%%%%%%%%%%%%%%%%%%%%%%%%%%%%%%%%%%%%%%%%%%%%%%%%%%%%%%%%%%%%%%%%%%%%%%%%%%%%%%%%%%%%%%%%%%%%%%%%%%%
%%%%%%%%%%%%%%%%%%%%%%%%%%%%%%%%%%%%%%%%%%%%%%%%%%%%%%%%%%%%%%%%%%%%%%%%%%%%%%%%%%%%%%%%%%%%%%%%%%%%%%%%%%%%%%%

\section{Actions on flag manifolds}

We start by recalling the definition of flag manifolds
and discuss basic properties of the standard actions on the flag manifolds.
We also discuss properties of dynamics for actions on projective spaces and existence/uniqueness of sinks. Throughout this section, $G$ is a connected semisimple Lie group without any assumptions on its rank.

\subsection{Flag manifolds}\label{sec:flag}
We fix a Cartan involution $\theta$ of $G$.
It determines the maximal compact subgroup $K$ of $G$ and 
the Cartan decomposition
\begin{equation}
\label{eq:cartan}
\mathfrak{g}=\hbox{Lie}(K)\oplus \mathfrak{p}
\end{equation}
of the Lie algebra $\mathfrak{g}$ of $G$. Let $\mathfrak{a}\subset \mathfrak{p}$ be a Cartan subalgebra (i.e., a maximal abelian subalgebra of $\mathfrak{p}$).
A non-zero linear form $\alpha\in \mathfrak{a}^*$ is a (restricted) root if the corresponding root space 
$$
\mathfrak{g}_\alpha=\{x\in \mathfrak{g}:\, [a,x]=\alpha(a)x\quad \hbox{for $a\in\mathfrak{a}$}\}
$$
is non-zero. We denote by $\Phi\subset \mathfrak{a}^*$ the set of (restricted) roots. Then
$$
\mathfrak{g}=\mathfrak{g}_0\oplus \sum_{\alpha\in\Phi} \mathfrak{g}_\alpha\quad\quad
\hbox{and}\quad\quad
\mathfrak{g}_0=\mathfrak{m}\oplus \mathfrak{a},
$$ 
where $\mathfrak{m}=\hbox{Lie}(K)\cap \mathfrak{g}_0$. 
We fix a set $\Delta \subset \Phi$ of simple roots and denote by 
$\Phi^+$ the corresponding subset of positive roots. 

We introduce the set of the standard parabolic subgroups $P_I$ of $G$ which are associated to subsets $I\subset \Delta$.
We denote by $\Phi^I\subset \Phi$ the subset of roots that are linear combinations of elements from $I$.
The standard parabolic subalgebra is defined by
$$
\mathfrak{p}_I=\mathfrak{m}_I+ \mathfrak{a}+ \mathfrak{n}_I,
$$
where 
$$
\mathfrak{m}_I=\mathfrak{m}+\mathfrak{a}+\sum_{\alpha\in \Phi^I} \mathfrak{g}_\alpha \quad
\quad\hbox{and}\quad\quad 
\mathfrak{n}_I=\sum_{\alpha\in \Phi^+\backslash \Phi^I}\mathfrak{g}_\alpha.
$$
We also set	
$$
\mathfrak{n}^-_I=\sum_{\alpha\in \Phi^-\backslash \Phi^I}\mathfrak{g}_\alpha.
$$
Then 
$$
\mathfrak{g}=\mathfrak{n}^-_I\oplus \mathfrak{p}_I.
$$
The standard parabolic subgroup $P_I$ is defined as the normaliser of $\mathfrak{p}_I$ in $\mathfrak{g}$.
A general parabolic subgroup is a subgroup of $G$ which is conjugate of one of the standard parabolic subgroups $P_I$. The flag manifolds are the homogeneous spaces 
$$
F_I=G/P_I.
$$
Let 
$$
N_I^-=\exp(\mathfrak{n}^-_I).
$$
Then it follows from the Bruhat decomposition that
$$
U_I=N_I^-P
$$ is an open dense subset of $F_I$, and moreover, the complement of 
$N_I^-P$ in $F_I$ is a finite union of analytic submanifolds of lower dimensions.

Let us describe how typical elements $g\in G$ act on the flag manifolds $F_I$.
Every $g\in G$ can be written uniquely as a commuting product 
\begin{equation}
\label{eq:g1}
g=g_cg_{nc}g_u,
\end{equation}
where $\hbox{Ad}(g_c)$ is semisimple and has all eigenvalues of modulus one, 
$\hbox{Ad}(g_{nc})$ is semisimple and has all eigenvalues real and positive, and $\hbox{Ad}(g_u)$ is unipotent.
Moreover, after taking a conjugation of $g$, we may assume that 
$g_{nc}=a\in \exp(\mathfrak{a}^+)$,
where 
$$
\mathfrak{a}^+=\{a\in \mathfrak{a}:\, \alpha(a)\ge 0\quad\hbox{for all $\alpha\in \Delta$}\}
$$
is the positive Weyl chamber in $\mathfrak{a}$.
Then the action of $g$ on the open cell $U_I\subset F_I$ can be described as follows: 
\begin{equation}
\label{eq:action}
g\cdot \exp(x)P_I=\exp\left(\sum_{\alpha\in \Phi^-\backslash \Phi^I} e^{\alpha(\log(a))}\hbox{Ad} (g_cg_u)x_\alpha\right)P_I \quad
\hbox{for $x=\sum_{\alpha\in \Phi^-\backslash \Phi^I}x_\alpha\in \mathfrak{n}^-_I.$}
\end{equation}
We recall that the element $g$ is called \emph{$\mathbb{R}$-regular} if the number of eigenvalues, counted with multiplicity, of $\hbox{Ad}(g)$ is minimal possible.
This condition is equivalent to $a$ being in the interior of $\mathfrak{a}^+$. 
We also use the following more general notion of regularity.

\begin{definition}
{\rm
	For $J\subset\Delta$, an element $g\in G$ is called \emph{$J$-regular} if
	it is of the form \eqref{eq:g1} with $g_{nc}$ being conjugated to $a\in \exp(\mathfrak{a}^+)$ such that $\alpha(\log(a))>0$ for all $\alpha\in J$. 
	In particular,
	$\Delta$-regular elements are precisely the $\R$-regular elements.
}
\end{definition}

Now suppose that the element  $g$ in \eqref{eq:action} is  $(\Delta\backslash I)$-regular.
Then $\alpha(\log (a))<0$ for all $\alpha\in \Phi^-\backslash \Phi^I$.
We observe that the map $\hbox{Ad} (g_cg_u)$ preserves the root spaces and has eigenvalues of absolute value one, and  $\|\hbox{Ad} (g_cg_u)^n\|$ grows at most polynomially as $n\to\infty$.
Hence, it follows that the identity coset $eP_I$
is a differentiable sink for $g$, and for every $z\in U_I$, $g^n z\to eP_I$ as $n\to \infty$. Since $U_I$ is dense, it also clear that this sink is unique.
We denote by $s_g\in F_I$ the sink of such element $g$.

The above discussion shows that every $(\Delta\backslash I)$-regular element has a sink in $F_I$. As we shall see in the next section, the converse is also true: 
if an element has a topological sink in $F_I$, then it is $(\Delta\backslash I)$-regular.

\subsection{Dynamics on projective spaces}\label{sec:proj}
It is convenient to study the action of $G$ on the flag manifolds $F_I$ by embedding $F_I$ 
in a product of suitable projective spaces. We consider the $G$-equivariant embedding
\begin{equation}
\label{eq:embed}
\iota_I:F_I\to \prod_{\alpha\in \Delta\backslash I} \mathbb{P}(V_\alpha): gP_I\mapsto (\sigma_\alpha(g)v_\alpha:\, \alpha\in \Delta\backslash I),
\end{equation}
introduced in \cite[Sec.~3]{Benoist}, which is defined using suitable irreducible representations 
$$
\sigma_\alpha:G\to \hbox{GL}(V_\alpha),\quad \alpha\in \Delta,
$$
and the highest weight vectors $v_\alpha\in V_\alpha$.
These representations have the property that the transformation $\sigma_\alpha(g)$ is proximal 
if and only if $g$ is $\{\alpha\}$-regular (see \cite[Sec.~2.5]{Benoist}).
We recall that a linear transformation is called \emph{proximal} if 
it has a unique eigenvalue of maximal modulus, and this eigenvalue has multiplicity one.

\begin{proposition}\label{p:sink}
Suppose that the action of an element $g\in G$ on the flag manifold $F_I$ has a topological sink. 
Then $g$ is $(\Delta\backslash I)$-regular. In particular, every topological sink
for the standard action on $F_I$ is also a differentiable sink, and this sink is unique.
\end{proposition}

\begin{proof}
We prove the proposition by considering the action on $\iota_I(F_I)$.
If the action of $(\sigma_\alpha(g))_{\alpha\in\Delta\backslash I}$ on $\iota_I(F_I)$ 
has a topological sink, then the actions of $\sigma_\alpha(g)$ on $\iota_\alpha (F_I)$, $\alpha\in\Delta\backslash I$,
also have topological sinks. We will show that then $\sigma_\alpha(g)$'s are proximal. In view of the above remark, this will imply the proposition.
	
We write the transformation $\sigma_\alpha(g)$ as a commuting product 
$$
\sigma_\alpha(g)=kau,
$$
where $k$ is semisimple with eigenvalues of absolute value one,  $a$ is semisimple with real positive eigenvalues, and $u$ is unipotent.
Let $s_\alpha\in \iota_\alpha (F_I)$ be a topological sink of $\sigma_\alpha(g)$.
We fix a $\sigma_\alpha(K)$-invariant metric on $\mathbb{P}(V_\alpha)$
and denote by $B_{\epsilon}(s_\alpha)$ the $\epsilon$-ball centred at $s_\alpha$ in $\iota_\alpha (F_I)$.
Then since $s_\alpha$ is a sink, there exists $\epsilon_0>0$ such that for every $\epsilon>0$ and all sufficiently large $n$,
we have 
$$
g^n(B_{\epsilon_0}(s_\alpha))\subset B_{\epsilon}(s_\alpha),
$$
and hence,
$$
(au)^n(B_{\epsilon_0}(s_\alpha))\subset B_{\epsilon}(k^{-n}s_\alpha).
$$
Passing to subsequence, we may assume that $k^{-n_i}s_\alpha\to s_\alpha$. 
Hence, it follows that for every $\epsilon>0$ and all sufficiently large $i$, we have 
\begin{equation}
\label{e:sink}
(au)^{n_i}(B_{\epsilon_0}(s_\alpha))\subset B_{2\epsilon}(s_\alpha).
\end{equation}
In this case, we say that $s_\alpha$ is a topological sink for the sequence of transformations
$(au)^{n_i}$.

%We pass from $V_\alpha$ to the smallest $\sigma_\alpha(g)$-invariant subspace which contains $s_\alpha$.

The transformation $au$ has real positive eigenvalues.
Let $\lambda$ be the maximal eigenvalue of this transformation, and let $V_\alpha(\lambda)$ be the corresponding Jordan subspace.
We denote by $\pi_\lambda:V_\alpha\to V_\alpha(\lambda)$ the projection
map defined by the Jordan decomposition of $au$. 
We claim that $\pi_\lambda(s_\alpha)\ne 0$. Indeed, suppose that $\pi_\lambda(s_\alpha)= 0$.
It follows from \eqref{e:sink} that
for sufficiently small neighbourhood $\mathcal{O}$ of identity in $G$ and all $v\in \sigma_\alpha(\mathcal{O})s_\alpha$, we have 
$$
(au)^{n_i}v\to s_\alpha.
$$ 
On the other hand, if $\pi_\lambda(v)\ne 0$, then since $\lambda$ is maximal,
$(au)^{n_i}v$ must converge to a point
in $\mathbb{P}(V_\alpha(\lambda))$ which contradicts our assumption that $\pi_\lambda(s_\alpha)= 0$. Hence, we conclude that 
$$
\sigma_\alpha(\mathcal{O})s_\alpha\subset \hbox{ker}(\pi_\lambda).
$$
Since $\iota_\alpha(F_I)=\sigma_\alpha(G)s_\alpha$ is an analytic submanifold of $\mathbb{P}(V_\alpha)$, this implies that
$$
\iota_\alpha(F_I)\subset \hbox{ker}(\pi_\lambda).
$$
However, this contradicts irreducibility of the representation $\sigma_\alpha$.
Hence, we conclude that $\pi_\lambda(s_\alpha)\ne 0$.

The transformation $au$  acts on $\mathbb{P}(V_\alpha(\lambda))$ as $u$.
We can write $u=\exp(X)$ for some nilpotent $X\in\hbox{End}(V_\alpha)$
preserving the Jordan decomposition. Then $u$ is contained in a unipotent one-parameter subgroup
$U=\{\exp(tX)\}_{t\in\R}$
of $\sigma_\alpha(G)$.
The projection map $\pi_\lambda:V_\alpha\to V_\alpha(\lambda)$
is equivariant with respect to the action of $U$.
We consider the action of $U$ on $S=\pi_\lambda(\iota_\alpha(F_I))$.
Then the point $s=\pi_\lambda(s_\alpha)$ is a topological sink for the sequence $u^{n_i}$.
It follows from Lemma \ref{l:unipotent} below that $S=\{s\}$.
If $\dim(V_\alpha(\lambda))>1$, then it would follow that 
$\iota_\alpha(F_I)$ is contained in a proper subspace of $V_\alpha$,
but this contradicts irreducibility of the representation $\sigma_\alpha$.
Hence, we conclude that $\dim(V_\alpha(\lambda))=1$, and $\sigma_\alpha(g)$ is proximal. This
implies that $g$ is $\alpha$-regular for all $\alpha\in\Delta\backslash I$, and completes the proof.
\end{proof}

The following lemma was used in the proof of the previous proposition.

\begin{lemma}\label{l:unipotent}
	Let $U=\{u_t\}_{t\in\R}$ be a one-parameter unipotent group of linear transformations of a vector space $V$, $S\subset \mathbb{P}(V)$
	a $U$-invariant connected analytic submanifold, and $s\in S$ a topological sink for a sequence $u_{n_i}$ satisfying $n_i\to\infty$. Then $S=\{s\}$.
\end{lemma}

\begin{proof}
Since $s$ is a topological sink for the sequence $u_{n_i}$, there exists a neighbourhood $W_0$ of $s$ in $S$
such that for every neighbourhood $W$ of $s$ in $S$ and all sufficiently large $i$,
\begin{equation}
\label{eq:connt}
u_{n_i}(W_0)\subset W.
\end{equation}
Without loss of generality, we may assume that $V=\left<S\right>$.
Moreover, since $S$ is an analytic submanifold, it follows that $V=\left<W_0\right>$.
We write $u_t=\exp(tX)$ for a nilpotent transformation $X\in\hbox{End}(V)$.
We suppose that $X\ne 0$ and take $\ell\ge 1$ such that $X^\ell\ne 0$ and $X^{\ell+1}= 0$.
Then 
$$
u_t=\sum_{j=0}^\ell \frac{t^jX^j}{j!}.
$$
There exists $w\in W_0$ such that $X^\ell w\ne 0$.
Then 
$$
u_tw\to [X^\ell w]\quad\hbox{in $\mathbb{P}(V)$ as $t\to -\infty$ and as $t\to +\infty$.}
$$
In particular, it follows that 
\begin{equation}
\label{eq:s}
s=[X^\ell w].
\end{equation}
It also follows that there exists $t_0\in \R$ such that for all $t<t_0$,
we have 
$$
u_tw\in W_0.
$$
Then we deduce from \eqref{eq:connt} that
for every neighbourhood $W$ of $s$ in $S$ and all sufficiently large $i$,
$$
u_{n_i}\cdot u_tw=u_{n_i+t} w\in W.
$$
Then taking $t=-n_i$, we conclude that $w=s$.
However, it follows from \eqref{eq:s} that $Xs=0$, so that $w\ne s$.
This contradiction implies that $X=0$ (that is, $U$ is trivial), and since $s$ is a sink, $S=\{s\}$.
\end{proof}

\section{Existence of many sinks}
\label{sec:many sinks}

In this section we establish abundance of differentiable sinks
for arbitrary smooth actions which are $C^0$-conjugate to the standard actions on flag manifolds. We only require existence of a single differentiable sink.

The following proposition will play a central role in the proof of our main result.
It might have other applications, and we emphasize that this result is applicable to any Zariski dense subgroup of a semisimple group without any rank assumptions. 

\begin{proposition}\label{manysinks}
Let $G$ be a semisimple real algebraic group, $F_I=G/P_I$ a flag manifold, and 
$\Gamma \subset G$ a Zariski dense subgroup of $G$.  We denote by $\rho_0$ the standard action of $\Gamma$ on $F_I$. Let $\rho$ be a $C^1$-action of $\Gamma$  on a manifold $M$,
and $\phi: F_I \to M$ is a $C^0$-conjugacy intertwining the actions $\rho _0$ and $\rho$.  
Suppose that there exists $\gamma_0\in \Gamma$ such that $\rho(\gamma_0)$ has a differentiable sink in $M$. Then there exists a Zariski dense subsemigroup $S \subset \Gamma$ such that  $\rho (\gamma)$  has a differentiable sink in $M$ for every $\gamma \in S$. 
\end{proposition} 

We will need a quantitative version of the proximal property discussed in Section \ref{sec:proj}.
Given a proximal linear transformation $g:V\to V$ of a vector space $V$,
we denote by $s_g\in \mathbb{P}(V)$ the direction corresponding to the maximal eigenvalue, and by $X_g^<\subset \mathbb{P}(V)$ the set of directions corresponding to the complementary $g$-invariant subspace. 
We fix standard metrics on the projective spaces $\mathbb{P}(V)$ and
set
\begin{align*}
b^\epsilon_g=\{x\in \mathbb{P}(V):\, d(x,s_g)\le \epsilon\}\quad\hbox{and}
\quad B^\epsilon_g=\{x\in \mathbb{P}(V):\, d(x,X_g^<)\ge \epsilon\}.
\end{align*}

\begin{definition}
	{\rm 
	We call a proximal transformation $g:V\to V$ \emph{$(r,\epsilon)$-proximal}
	if 
	$$
	d(s_g,X_g^<)\ge r,\quad g(B^\epsilon_g)\subset b^\epsilon_g,\quad 
	\hbox{ $g|_{B^\epsilon_g}$ is $\epsilon$-Lipschitz.}
	$$
	
	}
\end{definition}
This definition is slight variation of the notion introduced in \cite[Sec.~2-3]{Benoist}.
Adopting it to our setting, we say that 

\begin{definition}
	{\rm 
		 For $J\subset \Delta$, an element $g\in G$ is called \emph{$(J,r,\epsilon)$-regular}
		if the linear transformations $\sigma_\alpha(g)$  are \emph{$(r,\epsilon)$-proximal} for all $\alpha\in J$.
	}
\end{definition}

We use the following lemma \cite[3.6]{Benoist}. We fix standard metrics on the projective spaces $\mathbb{P}(V_\alpha)$ which also define a a metric on the flag manifolds $F_I$
via the embedding \eqref{eq:embed}.

\begin{lemma}[Benoist]   \label{benoist}
	Let $\gamma_0\in\Gamma$ be a $(\Delta\backslash I)$-regular element with the sink $s_{\gamma_0}\in F_I$.
	Then for every sufficiently small $r,\epsilon>0$, the set
	$$
	\mathcal{G}(\gamma_0,r,\epsilon)=\{\delta\in\Gamma:\, \hbox{$\delta$ is $(\Delta\backslash I,r,\epsilon)$-regular and }\; d(s_\delta, s_{\gamma_0})\le\epsilon\}
	$$
 is Zariski dense in $G$. 
\end{lemma}
This lemma is proved in \cite{Benoist} for $r=2\epsilon$, but the same argument allows to treat $(r,\epsilon)$-proximal elements as well.

\begin{proof}[Proof of Proposition \ref{manysinks}]
Throughout the proof, to simplify notation, for $\gamma\in \Gamma$ and  $x\in F_I$,
we write $\rho_0(\gamma)x=\gamma x$ for the standard action $\Gamma$ on $F_I$.

Let $m_0\in M$ be a differentiable sink for $\rho(\gamma_0)$.
We fix a Riemannian metric on $M$ such that 
\begin{equation}
\label{eq:diff_contr}
\|D(\rho(\gamma_0))_{m_0}\|<1.
\end{equation}
Since $\phi$ is a homeomorphism, it follows that $\phi^{-1}(m_0)$ is 
a topological sink for $\gamma_0$. Hence, it follows from Proposition \ref{p:sink}
that $\gamma_0$ is $(\Delta\backslash I)$-regular, and $\phi^{-1}(m_0)=s_{\gamma_0}$ is 
the differentiable sink for $\gamma_0$.

Since $\gamma_0$ is $(\Delta\backslash I)$-regular, there exists $r_0=r_0(\gamma_0)>0$ such that
for every $\lambda>0$ and 
every $n\ge n_0(\lambda)$, the element $\gamma_0^n$ is 
$(\Delta\backslash I,r_0,\lambda)$-regular.
We fix $r_0$ as above, and moreover, assume that it is sufficiently small,
so that the set $\mathcal{G}(\gamma_0,r_0,\epsilon)$ is Zariski dense
for all sufficiently small $\epsilon>0$ (see Lemma \ref{benoist}).

We claim that  for every $\kappa\in (0,1)$ and $\delta\in \mathcal{G}(\gamma_0,r_0,\epsilon)$, there exists $c=c(\delta)>0$ such that for $\epsilon\in (0,r_0/3)$ as above and $n\ge n_0(\kappa,\delta)$, 
we have 
\begin{align}
\delta \gamma _0 ^n \delta ^{-1} (B_{\epsilon/c} (\delta s_{\gamma _0})) &\subset B_{\epsilon/c} (\delta s_{\gamma _0}), \label{eq:1}\\
   \| D (\rho (\delta \gamma _0 ^n \delta ^{-1})_x  \| &\leq \kappa\quad\hbox{when $x\in \phi(B_{\varepsilon/c} (\delta s_{\gamma _0}))$.} \label{eq:2}
\end{align}
The quantity $n_0(\kappa,\delta)$ will be specified along the proof.

To prove \eqref{eq:1}, we choose $c=c(\delta)\ge 1$ so that
$$
d(\delta^{-1}x,\delta^{-1}y)\le c\,d(x,y)\quad\hbox{for all $x,y\in F_I.$}
$$
This implies that 
\begin{equation}
\label{eq:cont2}
\delta^{-1}(B_{\epsilon/c}(\delta s_{\gamma_0}))\subset 
B_{\epsilon}(s_{\gamma_0}).
\end{equation}
We take $\lambda=\min\{1/c,r_0/2\}$.
For every $n\ge n_0(\lambda)$, the element $\gamma_0^n$ is $(\Delta\backslash I,r_0,\lambda)$-regular.
Then since $\epsilon<r/3$, we have $r_0-\epsilon\ge \lambda$, and 
$$
\iota_\alpha(B_{\epsilon}(s_{\gamma_0}))\subset 
B^{\lambda}_{\sigma_\alpha(\gamma_0)}\quad\hbox{for all $\alpha\in\Delta\backslash I$}.
$$
Hence, since the transformations $\sigma_\alpha(\gamma_0^n)|_{B^{\lambda}_{\sigma_\alpha(\gamma_0)}}$,
$\alpha\in\Delta\backslash I$, are $\lambda$-Lipschitz, we conclude that
$$
\gamma_0^n(B_{\epsilon}(s_{\gamma_0}))\subset B_{\lambda\epsilon}(s_{\gamma_0})
\subset B_{\epsilon/c}(s_{\gamma_0}).
$$
Since $d(s_\delta,s_{\gamma_0})\le\epsilon$, we also have 
$$
B_{\epsilon/c}(s_{\gamma_0}) \subset B_{\epsilon/c+\epsilon}(s_{\delta})\subset B_{2\epsilon}(s_{\delta}),
$$
and since $\epsilon < r_0/3$,
$$
\iota_\alpha(B_{2\epsilon}(s_{\delta}))\subset 
B^{\epsilon}_{\sigma_\alpha(\delta)}\quad\hbox{for all $\alpha\in\Delta\backslash I$}.
$$
Using that the transformations $\sigma_\alpha(\delta)|_{B^{\epsilon}_{\sigma_\alpha(\delta)}}$,
$\alpha\in\Delta\backslash I$, are $\epsilon$-Lipschitz, we deduce that
$$
\delta(B_{\epsilon/c}(s_{\gamma_0}))\subset B_{\epsilon^2/c}(\delta s_{\gamma_0})
$$
which is contained in $B_{\epsilon/c}(\delta s_{\gamma_0})$ provided that $\epsilon\le 1$.
This completes the proof of \eqref{eq:1}.

Now we proceed with proving \eqref{eq:2}. It follows from \eqref{eq:cont2} that
\begin{equation}
\label{eq:step1}
\rho(\delta^{-1})(\phi(B_{\epsilon/c}(\delta s_{\gamma_0})))\subset 
\phi(B_{\epsilon}(s_{\gamma_0})).
\end{equation}
We take $\ell_1\ge 1$ such that $\gamma_0^{\ell_1}$ is $(\Delta\backslash I,r_0,\epsilon_0)$-regular
with $\epsilon_0=\min\{r_0,1\}/2$.
Then since we have assumed that $\epsilon<r_0/3$, we have 
$\epsilon_0<r_0-\epsilon$,  and  for all $\alpha\in\Delta\backslash I$,
$$
\iota_\alpha(B_{\epsilon}(s_{\gamma_0}))\subset B^{\epsilon_0}_{\sigma_\alpha(\gamma_0)}.
$$
This implies that 
\begin{equation}
\label{eq:cont0}
\gamma_0^{\ell_1}(B_{\epsilon}(s_{\gamma_0})) \subset  B_{\epsilon_0}(s_{\gamma_0}).
\end{equation}
Moreover, since the transformations $\sigma_\alpha(\gamma_0^{\ell_1})|_{B^{\epsilon_0}_{\sigma_\alpha(\gamma_0)}}$,
$\alpha\in\Delta\backslash I$, are $\epsilon_0$-Lipschitz, 
\begin{equation}
\label{eq:contr}
\gamma_0^{\ell_1} (B_{\theta}(s_{\gamma_0}))\subset B_{\epsilon_0\theta}(s_{\gamma_0})\quad \hbox{for every $\theta\in (0,\epsilon_0]$.}
\end{equation}
This implies that the sets $W_\theta=\phi(B_{\theta}(s_{\gamma_0}))$ give  
$\rho(\gamma_0^{\ell_1})$-invariant neighbourhoods of $m_0$ for $\theta\in (0,\epsilon_0]$.
We choose $\theta_0\in (0,\epsilon_0]$ sufficiently small so that 
$$
\eta=\sup_{x\in W_{\theta_0}} \|D(\rho(\gamma_0^{\ell_1}))_x\|<1.
$$
This is possible in view of \eqref{eq:diff_contr}.
Then it follows from the Chain Rule that there exists $C>0$ such that for every $n\ge 1$,
\begin{equation}
\label{eq:lambda}
\sup_{x\in W_{\theta_0}} \|D(\rho(\gamma_0^n))_x\|\le C\, \eta^{\lfloor n/\ell_1\rfloor }.
\end{equation}
Using \eqref{eq:contr}, we deduce that there exists $\ell_2\ge  1$ such that 
$$
\gamma_0^{\ell_2}(B_{\epsilon_0}(s_{\gamma_0}))\subset B_{\theta_0}(s_{\gamma_0}).
$$
Then it follows from \eqref{eq:cont0} that
$$
\gamma_0^{\ell_1+\ell_2}(B_{\epsilon}(s_{\gamma_0})) \subset  B_{\theta_0}(s_{\gamma_0}),
$$
and 
$$
\rho(\gamma_0^{\ell_1+\ell_2})(\phi(B_{\epsilon}(s_{\gamma_0}))) \subset  W_{\theta_0}.
$$
Hence, using \eqref{eq:step1}, we conclude that 
$$
\rho(\gamma_0^{\ell_1+\ell_2}\delta^{-1})(\phi(B_{\epsilon/c}(\delta s_{\gamma_0}))) \subset  W_{\theta_0}.
$$
Now the claim \eqref{eq:2} follows the estimate \eqref{eq:lambda} and the Chain Rule.

We consider the semigroup of the form
$$
S=\left< \delta \gamma_0^{n} \delta^{-1}: \, \delta\in \mathcal{G}(\gamma_0,r_0,\epsilon), \, n\ge n_0(\kappa, \delta)\right>. 
$$
It follows from property \eqref{eq:1} that for every $\gamma\in S$, the transformation $\rho(\gamma)$ preserves the neighbourhood $U= \phi(B_{\epsilon/c}(\delta s_{\gamma_0}))$.
Moreover, by \eqref{eq:2},
$$
\sup_{x\in U} \| D (\rho (\gamma))_x  \|\leq \kappa.
$$
When $\kappa$ is sufficiently small, this implies that $\rho(\gamma)$ is a contraction with respect to the Riemannian metric on $U$. Hence, we conclude that the map $\rho(\gamma)$ has a fixed point in $U$ which is a differentiable sink.

It remains to show that the semigroup $S$ is Zariski dense.
Let $\bar S$ be the Zariski closure of $S$.
We denote by $A_n$ the Zariski connected component of the Zariski closure of the cyclic group $\left<\gamma_0^n\right>$. We note that since $\gamma_0$ has a sink, it must be of infinite order, so that $A_n$ is not trivial for all $n$.
Moreover, we may assume that the projections of $\gamma_0$
to all non-trivial simple factors of $G$ also have infinite order.
Indeed, suppose that for some non-trivial simple factor $G_i\subset G$, the $G_i$-component of $\gamma_0$
has finite order. Then for some $n$, the transformation $\gamma_0^n$ acts trivially on the submanifold $G_is_{\gamma_0}\subset F_I$. Since $s_{\gamma_0}$ is a sink for $\gamma_0$,
this implies that $G_is_{\gamma_0}=s_{\gamma_0}$, and $G_i\subset P_I$.
In this case, we can replace the group $G$ by $G/G_i$. Hence, without loss of generality,
the projections of $\gamma_0$
to all non-trivial simple factors of $G$ also have infinite order.

We note that $A_m\supset A_n$ when $m$ divides $n$ and 
consider the descending sequence of Zariski closed subgroups $B_n=A_{n!}$. 
For sufficiently large $n$, this sequence
stabilises, and we denote the minimal element by $B$.
For every $\delta\in \mathcal{G}(\gamma_0,r,\epsilon)$, we have
$\delta B \delta^{-1} \subset \bar S$. 
Hence, it follows from Zariski density of $\mathcal{G}(\gamma_0,r,\epsilon)$ that 
$\bar S$ contains the conjugacy class $B^G$. 
Since $S$ is a semigroup, its Zariski closure $\bar S$ is a group. 
We conclude that $\bar S$ contains the normal subgroup generated by $B$.  Since the projections of $\gamma _0$ to all nontrivial simple factors of $G$ have infinite order, it follows 
 $\bar S=G$, so that $S$ is Zariski dense.
\end{proof}

%%%%%%%%%%%%%%%%%%%%%%%%%%%%%%%%%%%%%%%%%%%%%%%%%%%%%%%%%%%%%%%%%%%%%%%%%%%%%%%%%%%%%%%%%%%%%%%%%%%%%%%%%%%%%%%
%%%%%%%%%%%%%%%%%%%%%%%%%%%%%%%%%%%%%%%%%%%%%%%%%%%%%%%%%%%%%%%%%%%%%%%%%%%%%%%%%%%%%%%%%%%%%%%%%%%%%%%%%%%%%%%

\section{Projection maps}
\label{sec:projections} 

Let $G$ be a connected semisimple Lie group, $F_I=G/P_I$ a flag manifold, and 
$\Gamma \subset G$ a lattice subgroup of $G$.  We denote by $\rho_0$ the action of $\Gamma$ on $F_I$. Let $\rho$ be a smooth action of $\Gamma$  on a compact manifold $M$.
In this section we study properties of a $C^0$-conjugacy map
$$
\phi: F_I \to M
$$
that intertwines the actions $\rho _0$ and $\rho$. We assume that every simple factor of $G$
has real rank at least two. The aim of this section is to construct a family of projections maps
$\pi_0^{(\alpha)}:U_I \to U_I$ defined on the open cell $U_I=N_I^- P_I\subset F_I$
and the corresponding projection maps  $\tilde \pi_0^{(\alpha)}$ for $M$ which are conjugated to $\pi_0^{(\alpha)}$ via $\phi$.

Since the centre of $G$ acts trivially on the flag manifold $F_I$, we may assume without loss of generality that $G$ is centre-free. 
It follows from the Margulis Arithmetiticy theorem 
(see  \cite[Ch.~IX]{margulis-book} or \cite[Ch.~6]{zimmer}) that 
there exist a connected semisimple algebraic $\Q$-group $\bf G$ and surjective homomorphism 
$\iota:{\bf G}(\R)^\circ\to G$ such that $\hbox{ker}(\iota)$ is compact and 
$\iota({\bf G}(\Z)\cap {\bf G}(\R)^\circ)$ is commensurable to $\Gamma$.
Hence, without loss of generality, we may assume that 
$G={\bf G}(\R)^\circ$ and $\Gamma$ is a finite index subgroup of ${\bf G}(\Z)\cap {\bf G}(\R)^\circ$.

In order to have rich dynamics in a neighbourhood of a sink,
we need to construct commuting elements satisfying certain independence properties.
More precisely, these elements will be chosen in the centraliser $Z_\Gamma(\gamma_0)$
where $\gamma_0$ is picked from a given Zariski dense subsemigroup $S$.
Eventually, we apply this construction for the semigroup $S\subset \Gamma$ introduced in Proposition \ref{manysinks}, but the discussion in the first part of this section applies
to arbitrary Zariski dense subsemigroup $S\subset  {\bf G}(\Q)$.

Our argument is based on the results established by Prasad and Rapinchuk \cite{Prasad-Rapinchuk}. We also refer to \cite{Prasad-Raghunathan} for basic properties
of regular and $\R$-regular elements. We start by introducing required notation.
We denote by $Z_{\bf G}(g)^\circ$ the connected component of the centraliser of $g$ in $\bf G$
with respect to the Zariski topology.
% and by $Z_{G}(g)^\circ$ the connected component of the centraliser of $g$ in $G$ with respect to the Euclidean topology.
We recall that if $g\in G$ is a regular $\R$-regular element, then 
$Z_{\bf G}(g)^\circ$ is a maximal torus in $\bf G$, and 
\begin{equation}
\label{eq:centraliser}
Z_{\bf G}(g)^\circ={\bf B}_g {\bf T}_g,
\end{equation}
where ${\bf B}_g$ is a torus such that ${\bf B}_g(\R)$ is compact,
and ${\bf T}_g$ is a maximal $\R$-split torus in ${\bf G}$. 
We note that an $\R$-regular element is necessarily semisimple.
In particular, it follows from \cite[11.12]{Borel} that $g\in Z_{\bf G}(g)^\circ$.
For a character $\chi$ of ${\bf T}_g$ and $h=bt\in Z_{\bf G}(g)^\circ={\bf B} {\bf T}_g$,
we set $\alpha(h)=\alpha(t)$. Since $Z_{\bf G}(g)^\circ$ has finite index in $Z_{\bf G}(g)$,
it follows that for every $h\in Z_{\bf G}(g)$, $h^\ell\in Z_{\bf G}(g)^\circ$ for some $\ell\ge 1$.
In that case, we set $\chi(h)=\chi(h^\ell)^{1/\ell}$. It is clear that this definition
is independent of the choice of exponent $\ell$.

We say that $g\in {\bf G}(\Q)$ is \emph{anisotropic} if 
the torus $Z_{\bf G}(g)^\circ$ is anisotropic over $\Q$.

%Also,
%$$
%Z_{G}(g)^0={\bf B}(\R)^\circ {\bf T}_g(\R)^\circ,
%$$
%If we write $g$ as a commuting product $g=g_c g_{nc}$ 
%where $\hbox{Ad}(g_c)$ is semisimple and has all eigenvalues of modulus one, and 
%$\hbox{Ad}(g_{nc})$ is semisimple and has all eigenvalues real and positive,
%then $g_{nc}\in {\bf T}_g(\R)^\circ$. 
%and  where $A$ is suitable Cartan subgroup of $G$, and $B=Z_K(A)^0$ for a suitable maximal compact subgroup $K$ of $G$. Moreover, the element $g$ can be written as a commuting product $g=ka$ with $k\in K$ and $a\in A$. 
The group ${\bf G}$ has a decomposition 
as an almost direct product 
\begin{equation}
\label{eq:prod0}
{\bf G}={\bf G}^{(1)}\cdots {\bf G}^{(r)},
\end{equation}
where ${\bf G}^{(i)}$'s  are the connected $\Q$-simple subgroups of ${\bf G}$.
We say that an element $g\in {\bf G}$ is \emph{without components of finite order} if 
with respect to this decomposition,
$g=g_1\cdots g_r$ with all $g_i$ of infinite order.
A maximal $\Q$-subtorus ${\bf T}$ of ${\bf G}$ is called \emph{$\Q$-quasi-irreducible} if it does not contain any $\Q$-subtori
other than almost direct products of the tori ${\bf T}^{(i)}={\bf T}\cap {\bf G}^{(i)}$.

We say that commuting elements $\delta_1,\delta_2\in {\bf G}(\Q)$ are \emph{multiplicatively independent} if the projections of $\delta_1$ and $\delta_2$ on every non-trivial $\Q$-simple factor of ${\bf G}$ generate a subgroup isomorphic to $\Z^2$.

\begin{lemma}   \label{prasad-rapinchuk}
Let $S$ be a Zariski dense subsemigroup of ${\bf G}(\Q)$.
Then there exists a regular $\R$-regular element $\gamma_0\in S$ 
which is anisotropic, without components of finite order, and 
such that for every commuting multiplicatively independent
$\delta_1,\delta_2\in Z_{{\bf G}}(\gamma_0)(\Q)$ and every 
non-trivial $\R$-character $\chi$ of ${\bf T}_{\gamma_0}$, the real numbers $\chi(\delta_1)$ and $\chi(\delta_2)$ are multiplicatively independent.
\end{lemma}

\begin{proof}
It follows from \cite[Theorem 2]{Prasad-Rapinchuk} that there exists regular $\R$-regular $\gamma_0\in S$ such that $Z_{\bf G}(\gamma_0)^\circ$ is a $\Q$-quasi-irreducible torus in ${\bf G}$ which is anisotropic over $\Q$. Let us suppose that for some
non-trivial $\R$-character $\chi$ of ${\bf T}_{\gamma_0}$, the real numbers $\chi(\delta_1)$ and $\chi(\delta_2)$ are multiplicatively dependent, that is, there exists $(n_1,n_2)\in\Z^2\backslash \{(0,0)\}$ such that 
$$
\chi(\delta_1)^{n_1}\chi(\delta_2)^{n_2}=
\chi(\delta_1^{n_1}\delta_2^{n_2})=1.
$$
Replacing $\delta_1, \delta_2$ by $\delta_1^\ell, \delta_2^\ell$ for suitable $\ell \ge 1$,
we may assume without loss of generality that $\delta_1, \delta_2\in Z_{\bf G}(\gamma_0)^\circ$.
Then the subgroup $\left<\delta_1^{n_1}\delta_2^{n_2}\right>$ is contained in a proper subtorus
of $Z_{\bf G}(\gamma_0)^\circ$. Hence, its Zariski closure gives a proper $\Q$-subtorus of 
$Z_{\bf G}(\gamma_0)^\circ$. Since $Z_{\bf G}(\gamma_0)^\circ$ is $\Q$-quasi-irreducible,
it follows that the projection of $\delta_1^{n_1}\delta_2^{n_2}$ to one of the  non-trivial $\Q$-simple factors of ${\bf G}$ should have finite order. However, this contradicts the assumption that 
$\delta_1$ and $\delta_2$ are multiplicatively independent. Hence, we conclude that 
$\chi(\delta_1)$ and $\chi(\delta_2)$ are multiplicatively independent for all 
non-trivial $\R$-characters $\chi$.
\end{proof}

Given a regular $\R$-regular $g\in G$, we denote by $\Phi_g=\Phi({\bf T}_g, {\bf G})$ the root system arising from the action of ${\bf T}_g$ on the Lie algebra of  ${\bf G}$.
Once an ordering on $\Phi({\bf T}_g, {\bf G})$ is given, we define the set
 of positive roots $\Phi_g^+$, the set
 of negative roots $\Phi_g^-$, and the set of simple roots $\Delta_g\subset\Phi_g^+$.
%For $I\subset \Delta_g$, we denote by $\Phi_g^I$ the set of roots which are products of elements from $I$.

\begin{lemma}\label{l:roots}
Let	$\gamma _0 \in \Gamma$ be an element as in  Lemma \ref{prasad-rapinchuk}.
We fix an ordering on $\Phi_{\gamma_0}$ such that 
	\begin{equation}
	\label{eq:ordering}
	\alpha(\gamma_0)<1\quad\hbox{ for all $\alpha\in \Phi_{\gamma_0}^-$.}
	\end{equation}
	Then for every simple root $\alpha_0\in  \Delta_{\gamma_0}$, there exists a sequence $\delta_n \in \Gamma\cap Z_{\bf G}(\gamma_0)^\circ$ consisting of commuting $\R$-regular elements and satisfying
$$
 \alpha_0(\delta_n)\to 1,
$$
and 
$$
\alpha (\delta_n) \rightarrow 0\quad\hbox{for all $\alpha\in  \Phi_{\gamma_0}^-$ that are not proportional to $\alpha_0$.}
$$
\end{lemma}

\begin{proof}
Since the group ${\bf H}=Z_{\bf G}(\gamma_0)^\circ$ is anisotropic over $\Q$, 
it follows that ${\bf H}(\Z)$ is a lattice in ${\bf H}(\R)$. 
In particular, we deduce that $\Gamma\cap {\bf H}$ is a lattice in ${\bf H}(\R)$, and it contains a subgroup $\Lambda\simeq \Z^r$ where $r=\dim({\bf T}_{\gamma_0})$ is the $\R$-rank of ${\bf G}$.
It also follows that $\gamma_0^\ell \in \Lambda$ for some $\ell\ge 1$.
We have a decomposition 
$$
{\bf H}={\bf H}^{(1)}\cdots {\bf H}^{(r)},
$$
where 
${\bf H}^{(i)}={\bf H}\cap {\bf G}^{(i)}$ and ${\bf G}^{(i)}$'s are the connected
$\Q$-simple normal subgroups of $\bf G$ from \eqref{eq:prod0}. 
Since ${\bf H}^{(i)}$'s are anisotropic over $\Q$,
${\bf H}^{(i)}(\Z)$ is a lattice in ${\bf H}^{(i)}(\R)$ as well.
It follows from our assumption on the rank of $G$ that each of the factors ${\bf G}^{(i)}$ has $\R$-rank at least two, so that  
\begin{equation}
\label{eq:higher rank}
{\bf H}^{(i)}\cap \Lambda\simeq \Z^{r_i}\quad\hbox{ with $r_i\ge 2$.}
\end{equation}

We consider a collection of linear forms
$L_\alpha=\log (\alpha)$, $\alpha\in \Delta_{\gamma_0}$, on $\Z^r$
that defines the negative Weyl chamber
$$
C^-=\{a\in \Lambda\otimes \R:\, L_\alpha(a)< 0\hbox{ for all }\alpha\in \Delta_{\gamma_0}\}
$$
that contains $\gamma_0^\ell$.
Since $\Lambda$ is a lattice in $\Lambda\otimes \R$, there exists $\delta_0\in \Lambda$ such that 
\begin{equation}
\label{eq:weyl}
\alpha(\delta_0)\le \alpha_0(\delta_0)<0\quad\hbox{ for every $\alpha\in \Delta_{\gamma_0}$.}
\end{equation}
This, in particular, implies that $\delta_0$ has no components of finite order. 
It follows from \eqref{eq:higher rank} that there exists $\delta_1\in \Lambda$ such that
$\delta_0$ and $\delta_1$ are multiplicatively independent.
We consider the subgroup $\Lambda_0=\left<\delta_0,\delta_1\right>\simeq \Z^2$ of $\Lambda$.
We note that
$C^-\cap (\Lambda_0\otimes \R)$
defines a non-trivial cone in $\Lambda_0\otimes \R\simeq \R^2$.
It follows from Lemma \ref{prasad-rapinchuk} that $\hbox{ker}(L_\alpha|_{\Lambda_0})=0$, and 
forms $L_{\alpha_1}|_{\Lambda_0}$ and $L_{\alpha_2}|_{\Lambda_0}$ are proportional only when the roots $\alpha_1$ and $\alpha_2$ are proportional.
In particular, every non-trivial element of $\Lambda_0$ is $\R$-regular.
It follows from \eqref{eq:weyl} that the line $L_{\alpha_0}=0$ gives one of the faces of the cone $C^-\cap (\Lambda_0\otimes \R)$. Then there exists a sequence $\delta_n\in C^-\cap \Lambda_0$ such that $\delta_n\to \infty$ and $L_{\alpha_0}(\delta_n)\to 0$. 
Moreover, it clear that $|L(\delta_n)|\to \infty$ for any linear from $L$ on $\Lambda_0\otimes \R$ which is not proportional to $L_{\alpha_0}|_{\Lambda_0\otimes \R}$. In particular, 
$|L_{\alpha}(\delta_n)|\to \infty$ for any $\alpha\in \Delta_{\gamma_0}\backslash \{\alpha_0\}$.  Since $\delta_n\in C^-\cap \Lambda_0$, it follows that $L_{\alpha}(\delta_n)\to- \infty$ 
for all $\alpha\in \Delta_{\gamma_0}\backslash \{\alpha_0\}$. 
This also implies that $L_{\alpha}(\delta_n)\to- \infty$ 
for all $\alpha\in \Phi_{\gamma_0}^-$
that are not proportional to $\alpha_0$ and proves the lemma.
\end{proof}

\vspace{0.5cm}

Now we apply the above results to the Zariski dense subsemigroup $S$ of $\Gamma$
constructed in Proposition \ref{manysinks}.
We recall that for  every element $\gamma \in S$, the map $\rho (\gamma)$ has a differentiable sink in $M$. 
We fix $\gamma _0 \in S$ as in Lemma \ref{prasad-rapinchuk} that determines the ordering on $\Phi_{\gamma_0}$ satisfying \eqref{eq:ordering} and the corresponding set of simple roots $\Delta_{\gamma_0}\subset \Phi_{\gamma_0}$.
% and a root $\alpha_0\in  \Delta_{\gamma_0}$ such that $\alpha_0\notin \Phi_{\gamma_0}^I$.
%and a sequence $\delta^{(\alpha_0)}_n \in Z_{\Gamma}(\gamma _0)$ as in Lemma \ref{l:roots}.

It will be convenient to work with the Cartan decomposition \eqref{eq:cartan} which is 
compatible with $Z_{\bf G}(\gamma_0)^\circ$.
We choose a Cartan involution $\theta$ so that the torus $Z_{\bf G}(\gamma_0)^\circ$ is $\theta$-invariant.
Then 
$$
{\bf B}_{\gamma_0}(\R)^\circ\subset K\quad\hbox{and}\quad \hbox{Lie}({\bf T}_{\gamma_0}(\R))\subset \mathfrak{p}.
$$
Since ${\bf T}_{\gamma_0}$ is a maximal $\R$-split torus,
$A={\bf T}_{\gamma_0}(\R)^\circ$ gives a Cartan subgroup of $G$.
We abuse notation and identify the root system $\Phi_{\gamma_0}$ of ${\bf T}_{\gamma_0}$
with the root system of $\mathfrak{a}$ introduced in Section \ref{sec:flag}. 

Let $I\subset \Delta_{\gamma_0}$ and $\alpha_0\in \Delta_{\gamma_0}\backslash I$.
We recall that $U_I=N_I^-P_I$ denotes  the open cell in the flag manifold $F_I=G/P_I$.
We also denote by $\mathfrak{n}_I^{(\alpha_0)}$ the subalgebra of $\mathfrak{n}^-_I$
generated by the root spaces $\mathfrak{g}_\alpha$ with $\alpha\in\Phi^-_{\gamma_0}\backslash \Phi^I_{\gamma_0}$
which are proportional to $\alpha_0$, and define 
$$
N_I^{(\alpha_0)}=\exp\left(\mathfrak{n}_I^{(\alpha_0)}\right)\quad\hbox{and}\quad U_I^{(\alpha_0)}=N_I^{(\alpha_0)} P_I.
$$
The $U_I^{(\alpha_0)}$ is a non-trivial submanifold of the open cell $U_I$. Let 
$$
\pi^{(\alpha_0)}_0: U_I =\exp(\mathfrak{n}_I^-)P_I\to U_I^{(\alpha_0)}=\exp\left(\mathfrak{n}_I^{(\alpha_0)}\right)P_I
$$
be the natural projection map.

\begin{lemma}\label{l:proj}
	Let $\gamma_0\in \Gamma$ be as in Lemma \ref{prasad-rapinchuk}.
Then there exists a sequence $\delta^{(\alpha_0)}_n \in \Gamma\cap Z_{\bf G}(\gamma_0)^\circ$ consisting of commuting $\R$-regular elements and satisfying
\begin{equation}
\label{eq:alpha_1}
\alpha_0(\delta_n^{(\alpha_0)})\to 1,
\end{equation}
and 
\begin{equation}
\label{eq:alpha_2}
\alpha (\delta_n^{(\alpha_0)}) \rightarrow 0\quad\hbox{for all $\alpha\in  \Phi_{\gamma_0}^-$ that are not proportional to $\alpha_0$}
\end{equation}
such that the  projection map $\pi _0^{(\alpha_0)} $ is the limit of the maps $\rho _0 (\delta^{(\alpha_0)} _n)$ acting on $U_I\subset F_I$.

In particular, for every $g\in N_I^{(\alpha_0)}$ and $x\in U_I$,
\begin{equation}
\label{eq:comm}
\pi _0^{(\alpha_0)}(gx)= g\, \pi _0^{(\alpha_0)}(x),
\end{equation}
and for every $\gamma\in \Gamma\cap Z_{\bf G}(\gamma_0)^\circ$ and $x\in U_I$,
\begin{equation}
\label{eq:comm2}
\pi _0^{(\alpha_0)}(\rho_0(\gamma)x)= \rho_0(\gamma)\, \pi _0^{(\alpha_0)}(x).
\end{equation}
\end{lemma}

\begin{proof}
Let $\delta_n \in \Gamma\cap Z_{\bf G}(\gamma_0)^\circ$
be the sequence constructed in Lemma \ref{l:roots}.
Since $\gamma _0$ is regular and $\R$-regular, its centralizer satisfied \eqref{eq:centraliser},
so that we can write $\delta_n = k_n a_n$ with $k_n\in {\bf B}_{\gamma_0}(\R)$
and $a_n\in {\bf T}_{\gamma_0}(\R)$. For every $v\in \mathfrak{n}_I$,
$$
\rho_0(\delta_n) \exp(v)P_I=\exp(\hbox{Ad}(\delta_n)v)P_I=\exp(\hbox{Ad}(a_n)\hbox{Ad}(k_n)v)P_I.
$$
It follows from properties of the sequence $\delta_n$ that 
$\alpha_0(a_n)\to 1$ and $\alpha(a_n)\to 0$ for every $\Phi_{\gamma_0}^-\backslash \Phi_{\gamma_0}^I$, so that for every $\ell\ge 1$ and $w\in \mathfrak{n}_I$,
\begin{equation}
\label{eq:a_n}
\hbox{Ad}(a^\ell_n)w\to p_0^{(\alpha_0)}(w),
\end{equation}
where $p_0^{(\alpha_0)}:\mathfrak{n}_I^-\to \mathfrak{n}_I^{(\alpha_0)}$ is the natural projection map.
Moreover, this convergence is uniform over $\ell \ge 1$ and $w$ in compact sets.
We observe that the transformations $\hbox{Ad}(k_n)|_{\mathfrak{n}_I}$ belong to a compact abelian group $\Omega< \hbox{GL}(\mathfrak{n}_I)$.
Hence, passing to a subsequence, we may assume that $\hbox{Ad}(k_n) |_{\mathfrak{n}_I} \rightarrow \omega$ for some $\omega\in \Omega$.
For every $j\ge 1$,
there exists $\ell_j\ge 1$ such that 
$$
d(\omega^{\ell_j}, id) < j^{-1},
$$
and there exists $n_j\ge 1$ such that 
$$
d( \hbox{Ad}(k_{n_j})|_{\mathfrak{n}_I}, \omega) < (\ell_j j)^{-1}.
$$
Hence, it follows that 
$$
d( \hbox{Ad}(k_{n_j}^{\ell_j})|_{\mathfrak{n}_I}, id) < j^{-1}.
$$
Combining this estimate with \eqref{eq:a_n}, we deduce that for 
$v\in \mathfrak{n}_I$,
$$
\hbox{Ad}(\delta_{n_j}^{\ell_j})v=\hbox{Ad}(a^{\ell_i}_{n_j})\hbox{Ad}(k^{\ell_j}_{n_j})v\to p_0^{(\alpha_0)}(v).
$$
Hence, the required sequence can be taken to be $\delta_j^{(\alpha_0)}=\delta_{n_j}^{\ell_j}$.
This proves the first part of the lemma. The second part (equation \eqref{eq:comm})
also follows because for $v\in \mathfrak{n}^{(\alpha_0)}_I$ and $x\in U_I$,
$$
\rho_0(\delta^{(\alpha_0)}_n) \exp(v)x=\exp(\hbox{Ad}(\delta^{(\alpha_0)}_n)v) \,\rho_0(\delta^{(\alpha_0)}_n) x\to \exp(v)\,\pi_0^{(\alpha_0)}(x).
$$
The last claim (equation \eqref{eq:comm2}) is immediate as well because $Z_{\bf G}(\gamma_0)^\circ$ is commutative.
This completes the proof.
\end{proof}

Next we show that there exist dynamically defined projection maps on $M$ which are analogues of the projection maps $\pi_0^{(\alpha_0)}$. 
We recall that $\rho(\gamma_0)$ has a differentiable sink $s\in M$.
It is also clear that $s_0=eP_I$ is the unique sink of the $\R$-regular element $\gamma_0$.
Since the $C^0$-conjugacy $\phi:F_I\to M$ intertwines $\rho_0(\gamma_0)$ and $\rho(\gamma_0)$,
it follows that $\phi(s_0)=s$.

\begin{lemma}\label{l:proj2}
There exists a neighbourhood $\mathcal{O}_0$ of the sink $s$ such that 
%$\mathcal{O}_0$ is invariant under $\rho (\gamma_0)$ and $\rho (\delta^{(\alpha_0)}_n)$, and 
$\rho (\delta^{(\alpha_0)}_n) |_{\mathcal{O}_{0}}$ converges to a smooth map $\tilde{\pi}^{(\alpha_0)}_0 : \mathcal{O}_0 \rightarrow M$ satisfying
\begin{equation}
\label{eq:pi1}
\phi \circ \pi^{(\alpha_0)}_0 = \tilde{\pi}^{(\alpha_0)} _0 \circ \phi
\end{equation}
on $\phi^{-1}(\mathcal{O}_0)$.
\end{lemma}

To prove this lemma, we need to use the theory of polynomial normal forms for smooth diffeomorphisms, which we now recall. We refer to \cite{G,GK} for more details.
We note that we only require normal forms   at a differentiable sink, rather then the more elaborate theory of contractions on vector bundle extensions developed in \cite{G,GK}.  However, we are not aware of a simpler reference for our case.

Let $f$ be a diffeomorphism with a differentiable sink $s$,
$\chi_1,\dots\,\chi_l$ the different moduli of eigenvalues of $D(f) _{s}$, and  $m_1,\dots,m_l$
their multiplicities.  
We represent the tangent space $T_s(M) \simeq \mathbb R ^m$ as the
direct sum of the spaces  $\mathbb R ^{m_i},\dots, \mathbb R ^{m_l}$, and let
$(t_1,\dots,t_l)$ be the corresponding coordinate representation of
a vector $t \in \mathbb R^n$. Let 
$$
P: \R ^m \to \R ^m:\,\,(t_1,\dots,t_l)\mapsto
(P_1(t_1,\dots,t_l),\dots,P_l(t_1,\dots,t_l))
$$
be a polynomial
map preserving the origin. We will say that the map $P$ is of
\emph{subresonance type} if it contains only homogeneous terms
in $P_i(t_1,\dots,t_l)$ with degree of homogeneity $s_j$ in the
coordinates of $t_j,\,\,i=1,\dots,l$, for which the  subresonance
relations
$$
\chi_i \le \sum_{j\neq i}s_j\chi_j
$$
hold. There are only
finitely many subresonance relations and it is known (see \cite{G,GK}) that
polynomial maps of the subresonance type with invertible derivative
at the origin generate a finite-dimensional Lie group. We will denote
this group by $SR_{\chi}$. The polynomial maps of subresonance type provide 
convenient normal forms of the diffeomorphism $f$ and its centraliser.

\begin{proposition}[\cite{G,GK}] \label{normal form}
%	Let $\gamma$ be a diffeomorphism on a manifold $M$ with fixed point $p \in M$.  Suppose $p$ is a differentiable sink.  Then
	There exists a coordinate chart $\omega:\mathcal{O}\to \mathcal{O}'\subset \R^m$ around the sink $s$ for which $\omega\circ f\circ \omega^{-1}$  is a polynomial map of subresonance type 	contained in the group $SR_{\chi}$.
	Moreover, this  coordinate chart transforms into such
	a normal form in $SR_{\chi}$ any diffeomorphism which commutes with $f$.
\end{proposition}

With the help of this proposition, we prove Lemma \ref{l:proj2}.

\begin{proof}[Proof of Lemma \ref{l:proj2}] By Proposition \ref{normal form}, there is a neighbourhood $\mathcal{O}$ of $s=\phi (s_{0})$ on which we have normal forms for
every diffeomorphism that commutes with $\rho(\gamma _0)$ and, in particular,
for the diffeomorphisms $\rho (\delta^{(\alpha_0)}_n)$.
Replacing $\mathcal{O}$ by a neighbourhood $\mathcal{O}_0\subset \mathcal{O}$, we can assume that 
$\rho(\gamma_0)$ and $\rho (\delta^{(\alpha_0)}_n)$ map $\mathcal{O}_0$ into $\mathcal{O}$.  
Indeed, this follows from properties of the sequence
$\delta^{(\alpha_0)}_n$ (see \eqref{eq:alpha_1}--\eqref{eq:alpha_2}).
Then we define $\tilde{\pi}^{(\alpha_0)}_0 $ on $\mathcal{O}_0$  as the limit of $\rho (\delta^{(\alpha_0)}_n)$. 
It is clear that this limit exists in $C^0$-topology because $\phi$ intertwines the map $\rho (\delta^{(\alpha_0)} _n)$ with the map $\rho _0(\delta^{(\alpha_0)} _n)$.  
Moreover, since the maps $\rho (\delta^{(\alpha_0)}_n)|_{\mathcal{O}_0}$ are polynomials of bounded degree in the normal form coordinates, $\tilde{\pi}^{(\alpha_0)}_0 $ is also a polynomial of the same degree, hence smooth.  The equivariance property is immediate from the equivariance of $\phi$ 
with respect to the actions $\rho$ and $\rho _0$.   
\end{proof}

\begin{proposition}  \label{equivariant projection}
The local projection map $\tilde{\pi}^{(\alpha_0)}_0:\mathcal{O}_0\to M$,
defined in Lemma \ref{l:proj2}, extends to  a smooth map 
$\tilde{\pi}^{(\alpha_0)}_0:\phi(U_I)\to M$ so that
\begin{equation}
\label{eq:pi2}
\phi \circ \pi^{(\alpha_0)}_0 = \tilde{\pi}^{(\alpha_0)} _0 \circ \phi
\end{equation}
on $U_I$.
\end{proposition}

\begin{proof}
We can extend the smooth projection map $\tilde{\pi}^{(\alpha_0)}_0$ defined
on the neighbourhood of the sink $\mathcal{O}_0$ to the whole $\phi(U_I)$ using conjugation by $\rho(\gamma _0)$.  More precisely,
given any compact subset $\Omega$ of $\phi(U_I)$, there exists $n$ such that  
$\rho(\gamma _0) ^n (\Omega) \in \mathcal{O}_0$.  Then we set 
$$
\tilde{\pi}^{(\alpha_0)} _0 (x) = \rho(\gamma _0)^{-n} \,\tilde{\pi}^{(\alpha_0)} _0 (\rho( \gamma _0)^nx), \quad x\in \Omega.
$$
It follows from \eqref{eq:comm2} and \eqref{eq:pi1} that the definition of 
$\tilde{\pi}^{(\alpha_0)} _0$
is consistent with $\tilde{\pi}^{(\alpha_0)} _0$ on $\mathcal{O}_0$ and is independent of
the choice of $n$. This allows us to extend $\tilde{\pi}^{(\alpha_0)} _0$ to $\phi(U_I)$ 
so that \eqref{eq:pi2} holds.
\end{proof}

\begin{lemma}  \label{l:manifold}
There exists $\gamma \in \Gamma\cap Z_{\bf G}(\gamma_0)^\circ$ such that
 $\phi (U_I^{(\alpha _0)})$ in a neighbourhood of of the sink $s$ is equal to the strong stable manifold of $\rho(\gamma)$.  In particular, $\phi (U_I^{(\alpha _0)})$ is an immersed submanifold.
\end{lemma} 

\begin{proof}   
	As we already observed in the proof of Lemma \ref{l:roots},
	$\Gamma\cap Z_{\bf G}(\gamma_0)^\circ$ is a lattice subgroup in 
	$Z_{\bf G}(\gamma_0)^\circ(\R)$. Its projection to ${\bf T}_{\gamma_0}(\R)$ is also a lattice. Since the set of simple roots $\Delta_{\gamma_0}$ forms a basis of the dual space of ${\bf T}_{\gamma_0}(\R)^\circ$. It follows that there exists $\gamma\in \Gamma\cap Z_{\bf G}(\gamma_0)^\circ$ such that $\alpha_0(\gamma)<1$ and $\alpha(\gamma)>1$
	for all $\alpha\in \Delta_{\gamma_0}\backslash \{\alpha_0\}$. Then 
	$U_I^{(\alpha _0)}$ is the strong stable manifold of $\rho_0(\gamma)$ at $s_0=\phi^{-1}(s)$.
	Suppose $y$ is contained in a small neighbourhood of $s=\phi(s_0)$ in $M$, and it has its forward $\rho(\gamma)$-orbit in this neighbourhood. Let $y =\phi(x)$ for some $x \in U_I$.  Then the forward orbit $\rho_0 (\gamma) ^n x$ stays in  a small neighborhood of $ s_{0}$.  
	In particular, it follows that $x\in U_I^{(\alpha_0)}$ and $y\in \phi(U_I^{(\alpha_0)})$
	. Since $\rho _0 (\gamma)$ is purely hyperbolic at $ s_{0}$, we see that $\rho_0(\gamma) ^n y$ converges to $s_{0}$ exponentially fast.  Hence, for some fixed sufficiently large $n_0$,  the transformation $\rho _0 (\gamma ^{n_0} \gamma _0 ^{-1})$ also contracts $x$ to  $ s_{0}$.  Since $\phi$ is continuous, it follows that 
	$\rho(\gamma ^{n_0} \gamma _0 ^{-1})^n y=\rho(\gamma _0)^{-n}\rho(\gamma ^{n_0})^n  y $ converges to $s=\phi (s_{0})$.  On the other hand,  $s$ is a differentiable sink of 
	$\rho (\gamma _0)$, so that we conclude that $\rho(\gamma^{n_0}) ^n y$ must converge to $s$ exponentially fast. This proves that $\phi(U_I^{(\alpha_0)})$ is equal to 
	the strong stable manifold of $\rho(\gamma^{n_0})$ in a neighbourhood of $s$.
	In particular, we also derive that $\phi(U_I^{(\alpha_0)})$ is a submanifold 
	in a neighbourhood of $s$. We deduce that $\phi(U_I^{(\alpha_0)})$ is an immersed
	submanifold using the action of $\rho(\gamma_0)^{-1}$.
\end{proof}

\section{Smoothness along foliations}
\label{s:foliations}

We keep the notation from the previous section and continue our investigation
of the conjugacy map $\phi:F_I\to M$. In this section, we establish smoothness of $\phi$ restricted to
open dense subsets of submanifolds $U_I^{(\alpha_0)}$ of $F_I$. The higher rank assumption
on $G$ is crucial here.

Let $H_I^{(\alpha_0)}$ be a connected component of the closure of $\rho_0(Z_{\bf G}(\gamma _0)^\circ \cap \Gamma) |_{U_I^{(\alpha _0)}}$. The latter 
can be identified with the closure of the linear group $\hbox{Ad}(Z_{\bf G}(\gamma _0)^\circ \cap \Gamma) |_{\mathfrak{n}_I^{(\alpha _0)}}$. It follows from Lemma \ref{prasad-rapinchuk}
that $H_I^{(\alpha_0)}$ is not trivial.  Since $H_I^{(\alpha_0)}$ acts freely away from the fixed point $s_0$, we get a foliation of $U_I^{(\alpha _0)}\backslash \{s_0\}$ by $H_I^{(\alpha_0)}$-orbits. 
We also note the action of $H_I^{(\alpha_0)}$ commutes with the action of $\rho_0(\gamma_0)$.

\begin{proposition}    \label{smoothness along subfoliations}
The conjugacy map $\phi:F_I\to M$ restricted to $U_I^{(\alpha _0)}$
is $C ^{\infty}$ on an open dense subset.	
\end{proposition}

In the proof of this proposition, we use a technical lemma which involves 
equidistribution properties of flows on the homogeneous space $\Gamma\backslash G$.
This requires new notation which we now introduce.
Let 
$$
M=Z_K(\mathfrak{a})^\circ.
$$
Let $\{a_t\}$ be a one-parameter subgroup of the Cartan group $A$ such that 
$N_I^+$ is the contracting horospherical subgroup for it, namely,
$$
N_I^+=\{g\in G:\, a_t g a_t^{-1} \to e\}.
$$ 
We write $G$ as an almost direct product 
$$
G=G^{(1)}\cdots G^{(r)},
$$
where 
$G^{(i)}$'s are connected normal subgroup of $G$ such that $\Gamma\cap G^{(i)}$'s are  irreducible
lattices in $G^{(i)}$. Then $(\Gamma\cap G^{(1)})\cdots (\Gamma\cap G^{(r)})$
has finite index in $\Gamma$, and without loss of generality, we may assume that
$$
\Gamma=(\Gamma\cap G^{(1)})\cdots (\Gamma\cap G^{(r)}).
$$
We denote by $G_I$ the product of $G^{(i)}$'s which are contained in $P_I$.

\begin{lemma}\label{l:mixing}
	Given arbitrary  $g_0,g\in G$, $p_0\in MA$, a neighbourhood $\mathcal{O}$ of the identity in $G$, and a neighbourhood $\mathcal{O}'$ of identity in $N^+_I$,
	for all sufficiently large $t$,
	$$
	\Gamma\cap g_0\mathcal{O} G_I (\mathcal{O}'p_0 a_{t})^{-1} g^{-1}\ne \emptyset.
	$$
\end{lemma}

\begin{proof}
We first assume for simplicity that the lattice $\Gamma$ is irreducible.	
Then this lemma follows directly from well-known equidistribution properties of the expanding horospherical subgroups on the space $\Gamma\backslash G$ (see, for instance, \cite[2.2.1]{Kleinbock-Margulis}): for arbitrary $\psi\in C_c(\Gamma\backslash G)$, $f\in C_c(N^+_I)$ and $x\in \Gamma\backslash G$,
\begin{equation}
\label{eq:mixing}
\int_{N^+_I} f(p)\psi(xpa_t)\, dm_{N^+_I}(p)\to \left(\int_{N^+_I} f\, dm_{N^+_I}\right)\left(\int_{\Gamma\backslash G} \psi\, dm_{\Gamma\backslash G}\right)\quad\hbox{as $t\to\infty$},
\end{equation}
where $m_{N^+_I}$ denotes a Haar measure on $N_I^+$, and $m_{\Gamma\backslash G}$ is the Haar measure on $\Gamma\backslash G$ such that $m_{\Gamma\backslash G}(\Gamma\backslash G)=1$.
%We note that this result is stated in \cite{Kleinbock-Margulis} for 
%$f\in C_c(N^+_I)$, but the same proof also applies to $f\in C_c(P_I)$.
We take nonzero $f\ge 0$ with $\hbox{supp}(f)\subset p_0^{-1}\mathcal{O}'p_0$
(note that $MA$ normalises $N_I^+$), 
nonzero $\psi\ge 0$ with $\hbox{supp}(\psi)\subset \Gamma g_0\mathcal{O}$,
and $x=\Gamma g p_0$. It follows from \eqref{eq:mixing} that 
$$
\int_{N^+_I} f(p)\psi(xpa_t)\, dm_{N^+_I}(p)>0
$$
for all sufficiently large $t$.
This implies that there exists $p\in p_0^{-1}\mathcal{O}'p_0$ such that
$$
\Gamma gp_0 pa_t\in \Gamma g_0\mathcal{O}.
$$
Hence, it follows that 
$$
	\Gamma\cap g_0\mathcal{O}(\mathcal{O}'p_0 a_{t})^{-1} g^{-1}\ne \emptyset.
$$
This proves the lemma under the assumption that $\Gamma$ is irreducible.

Now we discuss the general case. We note that the proof of \eqref{eq:mixing} is based on
the mixing property on $\Gamma\backslash G$: for every $\psi_1,\psi_2\in C_c(\Gamma\backslash G)$,
\begin{equation}
\label{eq:mixing00}
\int_{\Gamma\backslash G} \psi_1(z)\psi_2(z g)\, dm_{\Gamma\backslash G}(z)\to \left(\int_{\Gamma\backslash G} \psi_1\, dm_{\Gamma\backslash G}\right) \left(\int_{\Gamma\backslash G} \psi_2\, dm_{\Gamma\backslash G}\right)
\quad \hbox{as $g\to \infty$ in $G$.}
\end{equation}
 Although \eqref{eq:mixing00} fails in general when the lattice $\Gamma$
is not irreducible, it is true provided that the projection of $g$ to every factor $G^{(i)}$
goes to infinity. We observe that the projection of $a_t$ to every simple factor which is not
contained in $G_I$ goes to infinity as $t\to \infty$. Hence, this flow still satisfies
the mixing property on the space $(\Gamma G_I)\backslash G$, or 
equivalently the mixing property for $G_I$-invariant functions on  the space $\Gamma\backslash G$.
Namely, 
for every $G_I$-invariant $\psi_1,\psi_2\in C_c(\Gamma\backslash G)$,
\begin{equation}
\label{eq:mixing000}
\int_{\Gamma\backslash G} \psi_1(z)\psi_2(z a_t)\, dm_{\Gamma\backslash G}(z)\to \left(\int_{\Gamma\backslash G} \psi_1\, dm_{\Gamma\backslash G}\right) \left(\int_{\Gamma\backslash G} \psi_2\, dm_{\Gamma\backslash G}\right)
\quad \hbox{as $t\to\infty$.}
\end{equation}
This allows us to use the same argument as in
\cite[2.2.1]{Kleinbock-Margulis} to deduce that for a $G_I$-invariant $\psi\in C_c(\Gamma\backslash G)$, $f\in C_c(N^+_I)$ and $x\in \Gamma\backslash G$,
\begin{equation}
\label{eq:mixing}
\int_{N^+_I} f(p)\psi(xpa_t)\, dm_{N^+_I}(p)\to \left(\int_{N^+_I} f\, dm_{N^+_I}\right)\left(\int_{\Gamma\backslash G} \psi\, dm_{\Gamma\backslash G}\right)\quad\hbox{as $t\to\infty$}.
\end{equation}
Now we can finish the proof as in the previous paragraph by taking 
nonzero $G_I$-invariant function $\psi\ge 0$ with $\hbox{supp}(\psi)\subset \Gamma g_0\mathcal{O} G_I$.
We deduce that for all sufficiently large $t$,
there exists $p\in p_0^{-1}\mathcal{O}'p_0$ such that
$$
\Gamma gp_0 pa_t\in \Gamma g_0\mathcal{O}G_I.
$$
This implies the lemma.
\end{proof}

\begin{proof}[Proof of Proposition \ref{smoothness along subfoliations}]
	First, we show that 	
	the conjugacy map $\phi:F_I\to M$ is $C^{\infty}$ along the $H_I^{(\alpha_0)}$-orbits in $U_I^{(\alpha _0)}\backslash\{s_{0}\}$. 
	As before, for a sequence $\gamma_n \in Z_{\bf G}(\gamma _0)^\circ\cap \Gamma$ 
	such that $\rho_0(\gamma_n)|_{U_I^{(\alpha_0)}}$ converges to $h \in H_I^{(\alpha_0)}$,  we get convergence in $C^0$-topology of the maps $\rho (\gamma _n)$ on the intersection of $\phi (U_I^{(\alpha _0)}\backslash \{s_{0}\})$ with a small enough neighbourhood contained in a normal forms coordinate chart of the sink $s=\phi(s_0)$ of $\rho (\gamma_0)$.  Since the normal forms of $\rho(\gamma_n)$'s are polynomials of fixed degree, the limiting map is also a polynomial of the same degree, and hence smooth.  Thus, we get an action of $H_I^{(\alpha_0)}$ on this neighbourhood which intertwines via $\phi$ with the action of $H_I^{(\alpha_0)}$ on $U_I^{(\alpha _0)}\backslash\{s_{0}\}$.  Since $H_I^{(\alpha_0)}$ commutes with $\gamma _0$, we can extend the action to a smooth action on $\phi(U_I^{(\alpha _0)}\backslash\{s_{0}\})$ which is again intertwined with the action of $H_I^{(\alpha_0)}$ on $U_I^{(\alpha _0)}\backslash \{s_{0}\}$ via $\phi$.  
	%Since $H_I^{(\alpha_0)}$ is Zariski connected, it has finitely many connected components as a Lie group.  
	Hence, it follows that $\phi$ is smooth along the $H_I^{(\alpha_0)}$-orbits in $U_I^{(\alpha _0)}\backslash\{s_{0}\}$.  

Our next step is, starting with the orbit folliation of $H_I^{(\alpha_0)}$,
to construct additional smooth foliations $\mathcal{F}_i$ 
defined on open subsets of $U_I^{(\alpha _0)}$.
If we show that for $x\in F_I$, 
\begin{equation}\label{eq:smooth}
\hbox{$\phi|_{\mathcal{F}_i(x)}$ is smooth,}
\end{equation}
and 
\begin{equation}
\label{eq:trans}
T_x (\mathcal{F}_1(x))+\cdots +T_x (\mathcal{F}_r(x))=T_x(U_I^{(\alpha_0)}),
\end{equation}
then it will follow that $\phi$ is smooth in a neighbourhood of $x$ in $U_I^{(\alpha_0)}$.
These new foliations are constructed as 
\begin{equation}
\label{eq:f_g}
\mathcal{F}_\gamma (y)=\pi_0^{(\alpha_0)}(\gamma H^{(\alpha _0)}_I y),\quad y\in U_I^{(\alpha _0)}\backslash \{s_0\},
\end{equation}
for suitable $\gamma\in\Gamma$.
We note that since $\phi$ is $\Gamma$-equivariant and also equivariant with respect to
$\pi^{(\alpha_0)}_0$ and $\tilde{\pi}^{(\alpha_0)}_0$ (see Lemma \ref{equivariant projection}), it is clear that \eqref{eq:smooth} holds. Hence, the main task is to arrange the transversality property \eqref{eq:trans}. We construct such foliations inductively.

We take arbitrary $g_0\in N_I^{(\alpha_0)}\backslash \{e\}$ 
and a neighbourhood $\mathcal{O}$ of identity in $G$ and consider a distribution 
$\mathcal{E}$ in $T(F_I)$ defined on the neighbourhood $\pi^{(\alpha_0)}_0(g_0\mathcal{O}P_I)$ of $g_0 P_I$ in $U_I^{(\alpha_0)}$ 
and contained in the tangent distribution of $U_I^{(\alpha_0)}$.
We start with $\mathcal{E}$ being the tangent distribution
to the orbit foliation $H_I^{(\alpha_0)}x$, $x\in U_I^{(\alpha_0)}\backslash \{s_0\}$.
If $\dim(U_I^{(\alpha_0)})=\dim(H_I^{(\alpha_0)})$, then the proposition follows 
from smoothness of $\phi$ along the orbit foliation, so that
we assume that $\mathcal{E}$ is not equal
to the tangent distribution of for $U_I^{(\alpha_0)}$.
For $g$ in the neighbourhood of $g_0$ in $N_I^{(\alpha_0)}$, we 
consider a family of subspaces 
$$
\mathcal{V}(g)=D(g)^{-1}_{eP_I}\mathcal {E}(gP_I)
$$
of the tangent space $T_{eP_I}(F_I)$. 
We make the identification $T_{eP_I}(F_I)\simeq \mathfrak{g}/\mathfrak{p}_I$,
so that $D(p)_{eP_I}=\hbox{Ad}(p)$ for $p\in P_I$.

We recall that $N_I^{(\alpha_0)}=\exp(\mathfrak{n}_I^{(\alpha_0)})$,
where $\mathfrak{n}_I^{(\alpha_0)}$ is the span of root spaces for negative roots
proportional to $\alpha_0$. It follows from the properties of root systems that either 
$$
N_I^{(\alpha_0)}=\exp(\mathfrak{g}_{-\alpha_0})\quad\hbox{or}\quad
N_I^{(\alpha_0)}=\exp(\mathfrak{g}_{-\alpha_0}+\mathfrak{g}_{-2\alpha_0}).
$$
We treat these two cases separately.

\vspace{0.5cm}

\noindent \underline{Case 1}: $N_I^{(\alpha_0)}=\exp(\mathfrak{g}_{-\alpha_0})$.
Then $N_I^{(\alpha_0)}$ is commutative, and 
$$
\mathcal{E}(g_0 P_I)\subseteq T_{g_0 P_I}(\exp(\mathfrak{g}_{-\alpha_0})g_0 P_I),
$$
so that
$$
\mathcal{V}(g_0) \subseteq D(g_0)_{eP_I}^{-1}T_{g_0 P_I}(\exp(\mathfrak{g}_{-\alpha_0})g_0 P_I)
=T_{eP_I}(\exp(\mathfrak{g}_{-\alpha_0})P_I)=\mathfrak{g}_{-\alpha_0}+\mathfrak{p}_I.
$$
In particular, it follows that $\hbox{Ad}(N_I^+)$ acts trivially on $\mathcal{V}(g_0)$, and $\hbox{Ad}(A)$ acts on $\mathcal{V}(g_0)$ by scalar multiples.
By Lemma \ref{l:mixing}, 
given arbitrary $p_0\in MA$, 
a neighbourhood $\tilde{\mathcal{O}}\subset \mathcal{O}$ of identity in $G$,
a neighbourhood $\mathcal{O}'$ of identity in $N_I^+$
and sufficiently large $t$, 
there exists 
$\gamma \in \Gamma$ such that 
\begin{equation}
\label{eq:mixing2}
\gamma\in g_0\tilde{\mathcal{O}} G_I(\mathcal{O}'p_0 a_{t})^{-1} g_0^{-1}.
\end{equation}
In particular, it follows that $\gamma g_0P_I\in g_0\mathcal{O}P_I$.
We claim that there exists $\gamma\in\Gamma$ such that
\begin{equation}
\label{eq:gamma}
\gamma g_0P_I\in g_0\mathcal{O}P_I\quad \hbox{and} \quad D(\pi^{(\alpha_0)}_0\circ \gamma)_{g_0P_I}\mathcal {E}(g_0P_I)\nsubseteq \mathcal {E}(\pi^{(\alpha_0)}_0(\gamma g_0 P_I)).
\end{equation}
It follows from \eqref{eq:mixing2} that
given arbitrary $p_0\in M$,
there exists a sequence $\gamma_i\in \Gamma$ such that
\begin{equation}
\label{eq:gamma_i}
\gamma_i=g_i(u_i p_0 a_{t_i})^{-1}g_0^{-1}\quad\quad\hbox{with}\;\; g_i P_I\to g_0 P_I,\;\;\;\; u_i\in N_I^+,\; u_i\to e,\;\;\;\; t_i\to \infty.
\end{equation}
Then $\gamma_i g_0 P_I=g_iP_I\to g_0P_I$, so that for sufficiently large $i$, we have 
$\gamma_i g_0 P_I\in g_0\mathcal{O}P_I$.
Suppose that for those $i$'s, 
we have 
$$
D(\pi^{(\alpha_0)}_0\circ \gamma_i)_{g_0 P_I}\mathcal {E}(g_0 P_I)\subset \mathcal {E}(\pi^{(\alpha_0)}_0(g_iP_I)).
$$
This means that 
$$
D(\pi^{(\alpha_0)}_0\circ g_i)_{eP_I} \hbox{Ad}((u_ip_0 a_{t_i})^{-1})\mathcal {V}(g_0)\subset \mathcal {E}(\pi^{(\alpha_0)}_0(g_iP_I)).
$$
We have 
$$
\hbox{Ad}(( u_i p_0 a_{t_i})^{-1}) \mathcal {V}(g_0)=\hbox{Ad}(p_0^{-1})\mathcal {V}(g_0).
$$
Hence, we conclude that for all $p_0\in M$,
$$
D(\pi^{(\alpha_0)}_0\circ g_0)_{eP_I} \hbox{Ad}(p_0^{-1})\mathcal {V}(g_0)\subset \mathcal {E}(g_0P_I).
$$
Since $g_0\in N_I^{(\alpha_0)}$, it follows from \eqref{eq:comm} that
$$
D(g_0)_{eP_I} D(\pi^{(\alpha_0)}_0)_{eP_I} \hbox{Ad}(p_0^{-1})\mathcal {V}(g_0)\subset \mathcal {E}(g_0P_I).
$$
Let $\mathcal{W}$ be the subspace generated by $\hbox{Ad}(p_0^{-1})\mathcal {V}(g)$, $p_0\in M$.
Then
$$
D(g_0)_{eP_I} D(\pi^{(\alpha_0)}_0)_{eP_I} \mathcal {W}\subset \mathcal {E}(g_0P_I),
$$
and 
$$
D(\pi^{(\alpha_0)}_0)_{eP_I} \mathcal {W}\subset \mathcal {V} (g_0).
$$
Since $\mathcal{W}$ is $\hbox{Ad}(M)$-invariant, it follows from Lemma \ref{l:m} below that
$\mathcal{W}=\mathfrak{g}_{-\alpha_0}+\mathfrak{p}_I$, and $\dim(\mathcal{E})=\dim (U_I^{\alpha_0})$, but we have assumed that
the distribution $\mathcal{E}$ is not equal to the tangent distribution of $U_I^{(\alpha_0)}$.
This contradiction shows that \eqref{eq:gamma} holds.
Therefore, we obtain a new distribution 
$D(\pi^{(\alpha_0)}_0\circ \gamma)\mathcal {E}$ contained in $T(U^{(\alpha_0)}_I)$ and 
defined in a neighbourhood $\pi^{(\alpha_0)}_0(\gamma g_0 P_I)\in g_0\mathcal{O}P_I$.
Moreover, the distribution
$D(\pi^{(\alpha_0)}_0\circ \gamma)\mathcal {E}+\mathcal{E}$
is strictly larger than the distribution $\mathcal{E}$.
Now we can apply the above argument to the distribution 
$D(\pi^{(\alpha_0)}_0\circ \gamma)\mathcal {E}+\mathcal{E}$ defined in a neighbourhood 
of $\pi^{(\alpha_0)}_0(\gamma g_0 P_I)$ contained in $\pi^{(\alpha_0)}_0(g_0\mathcal{O}P_I)$,
and doing this inductively,
we conclude that there exist $\gamma_1=e,\gamma_2,\ldots,\gamma_r\in \Gamma$ and $x\in \pi^{(\alpha_0)}_0(g_0\mathcal{O}P_I)$
such that 
\begin{equation}
\label{eq:eee}
D(\pi^{(\alpha_0)}_0\circ \gamma_1)\mathcal {E}+\cdots+D(\pi^{(\alpha_0)}_0\circ \gamma_r)\mathcal {E}=T_x(U_I^{(\alpha_0)}).
\end{equation}

\vspace{0.5cm}

\noindent\underline{Case 2(a)}:
 $N_I^{(\alpha_0)}=\exp(\mathfrak{g}_{\alpha_0}+\mathfrak{g}_{2\alpha_0})$.
Then
$$
\mathcal{E}(gP_I)\subseteq T_{gP_I}(\exp(\mathfrak{g}_{-\alpha_0}+\mathfrak{g}_{-2\alpha_0})gP_I),
$$
and
\begin{align*}
\mathcal{V}(g) &\subseteq D(g)_{eP_I}^{-1}T_{gP_I}(\exp(\mathfrak{g}_{-\alpha_0}+\mathfrak{g}_{-2\alpha_0})gP_I)
=T_{eP_I}(\exp(\mathfrak{g}_{-\alpha_0}+\mathfrak{g}_{-2\alpha_0})P_I)\\
&=\mathfrak{g}_{-\alpha_0}+\mathfrak{g}_{-2\alpha_0} +\mathfrak{p}_I.
\end{align*}
We also observe that when $\mathcal{E}$ is equal to the tangent distribution
of the orbit foliation of $H_I^{(\alpha_0)}$, then
it follows from the definition of the group $H_I^{(\alpha_0)}$ that
$$
\mathcal{E}(g_0P_I)\nsubseteq T_{g_0P_I}(\exp(\mathfrak{g}_{-2\alpha_0})g_0P_I).
$$
Hence,
$$
\mathcal{V}(g_0) \nsubseteq D(g_0)_{eP_I}^{-1}T_{g_0P_I}(\exp(\mathfrak{g}_{-2\alpha_0})g_0P_I).
$$
Since $g_0\in N_I^{(\alpha_0)}=\exp(\mathfrak{g}_{-\alpha_0}+\mathfrak{g}_{-2\alpha_0})$,
we have $g_0^{-1}\exp(\mathfrak{g}_{-2\alpha_0})=\exp(\mathfrak{g}_{-2\alpha_0})g_0^{-1}$.
Thus, we conclude that 
\begin{equation}
\label{eq:transverse}
\mathcal{V}(g_0) \nsubseteq T_{eP_I}(\exp(\mathfrak{g}_{-2\alpha_0})P_I)=\mathfrak{g}_{-2\alpha_0}+\mathfrak{p}_I.
\end{equation}
Similar argument also gives that 
\begin{equation}
\label{eq:transverse2}
\mathcal{V}(g_0) \nsubseteq \mathfrak{g}_{-\alpha_0}+\mathfrak{p}_I.
\end{equation}

In Case 2(a), we additionally assume that
\begin{equation}
\label{eq:v_g_0}
\mathfrak{g}_{\alpha_0}+\mathfrak{p}_I \nsubseteq \mathcal{V}(g_0).
\end{equation}
We claim that there exists $\gamma\in\Gamma$ such that
\begin{equation}
\label{eq:gamma2}
\gamma g_0P_I\in g_0\mathcal{O}P_I\quad \hbox{and} \quad D(\pi^{(\alpha_0)}_0\circ \gamma)_{g_0P_I}\mathcal {E}(g_0P_I)\nsubseteq \mathcal {E}(\pi^{(\alpha_0)}_0(\gamma g_0 P_I)).
\end{equation}
As in Case 1, we show that 
given arbitrary $p_0\in MA$,
there exists $\gamma_i\in \Gamma$ such that \eqref{eq:gamma_i} holds.
Then $\gamma_i g_0 P_I=g_iP_I\to g_0P_I$, so that for sufficiently large $i$, we have 
$\gamma_i g_0 P_I\in g_0\mathcal{O}P_I$.
Suppose that for those $i$'s, 
we have 
$$
D(\pi^{(\alpha_0)}_0\circ \gamma_i)_{g_0P_I}\mathcal {E}(g_0 P_I)\subset \mathcal {E}(\pi^{(\alpha_0)}_0(g_iP_I)).
$$
This means that 
$$
D(\pi^{(\alpha_0)}_0\circ g_i)_{eP_I} \hbox{Ad}((u_i p_0 a_{t_i})^{-1})\mathcal {V}(g_0)\subset \mathcal {E}(\pi^{(\alpha_0)}_0(g_iP_I)).
$$
We have 
\begin{align*}
\hbox{Ad}(( u_i p_0 a_{t_i})^{-1}) \mathcal {V}(g_0)
&= \hbox{Ad}( a_{t_i}^{-1})\hbox{Ad}(p_0^{-1})\hbox{Ad}(u_i^{-1}) \mathcal {V}(g_0)\\
&= \hbox{Ad}(p_0^{-1}) \hbox{Ad}( a_{t_i}^{-1})\hbox{Ad}(u_i^{-1}) \mathcal {V}(g_0).
\end{align*}
We observe that for $x\in \mathfrak{n}_I^+$, $v_1\in \mathfrak{g}_{-\alpha_0}$, and 
$v_2\in \mathfrak{g}_{-2\alpha_0}$,
\begin{equation}
\label{eq:comp1}
\hbox{Ad}(\exp(x))(v_1+v_2+\mathfrak{p}_I)=(v_1+[x,v_2])+v_2+\mathfrak{p}_I,
\end{equation}
and 
\begin{equation}
\label{eq:comp2}
\hbox{Ad}(a_{t}^{-1})(v_1+v_2+\mathfrak{p}_I)= e^{-\theta_0 t} v_1+e^{-2\theta_0 t} v_2+\mathfrak{p}_I
\end{equation}
for some $\theta_0>0$. 
It follows from \eqref{eq:transverse} that there exists $v=v_1+v_2\in \mathfrak{g}_{-\alpha_0}\oplus \mathfrak{g}_{-2\alpha_0}$ with $v_1\ne 0$
such that $v+\mathfrak{p}_I\in \mathcal{V} (g_0)$.  
Then using \eqref{eq:comp1} and \eqref{eq:comp2}, we deduce that
$$
\frac{\hbox{Ad}(a_{t_i}^{-1})\hbox{Ad}(u_i^{-1}) v}{\|\hbox{Ad}(a_{t_i}^{-1})\hbox{Ad}(u_i^{-1}) v\|}\to \frac{v_1}{\|v_1\|}\in \mathfrak{g}_{-\alpha_0}.
$$
Moreover, we observe that this limit is independent of the sequences $t_i$ and $u_i$.
Therefore, we conclude that every limit point of the sequence subspaces 
$\hbox{Ad}( a_{t_i}^{-1})\hbox{Ad}(u_i^{-1})\mathcal {V}(g)\to \mathcal{V}_\infty$
contains a non-trivial, independent of $p_0$
subspace $\mathcal{V}_\infty \subset \mathfrak{g}_{-\alpha_0}+\mathfrak{p}_I$.
We conclude that for all $p_0\in MA$,
$$
D(\pi^{(\alpha_0)}_0\circ g_0)_{eP_I} \hbox{Ad}(p_0^{-1})\mathcal {V}_\infty\subset \mathcal {E}(g_0P_I).
$$
Since $g_0\in N_I^{(\alpha_0)}$, it follows from \eqref{eq:comm} that also 
$$
D(g_0)_{eP_I} D(\pi^{(\alpha_0)}_0)_{eP_I} \hbox{Ad}(p_0^{-1})\mathcal {V}_\infty\subset \mathcal {E}(g_0P_I).
$$
Let $\mathcal{W}$ be the subspace generated by $\hbox{Ad}(p_0^{-1})\mathcal {V}_\infty$, $p_0\in MA$. Then
$$
D(g_0)_{eP_I} D(\pi^{(\alpha_0)}_0)_{eP_I} \mathcal {W}\subset \mathcal {E}(g_0P_I),
$$
and 
$$
D(\pi^{(\alpha_0)}_0)_{eP_I} \mathcal {W}\subset \mathcal {V} (g_0).
$$
Since $\mathcal{W}\subset \mathfrak{g}_{-\alpha_0}+\mathfrak{p}_I$ is $\hbox{Ad}(M)$-invariant, 
it follows from  Lemma \ref{l:m} below that 
$$
\mathcal{W}=\mathfrak{g}_{-\alpha_0}+\mathfrak{p}_I.
$$
However, this contradicts our assumption \eqref{eq:v_g_0}. Hence, we conclude that \eqref{eq:gamma2} holds. 
This gives a new distribution 
$D(\pi^{(\alpha_0)}_0\circ \gamma)\mathcal {E}$ contained in $T(U^{(\alpha_0)}_I)$ and 
defined in a neighbourhood $\pi^{(\alpha_0)}_0(\gamma g_0 P_I)\in g_0\mathcal{O}P_I$.
If the distribution $D(\pi^{(\alpha_0)}_0\circ \gamma)\mathcal {E}+\mathcal{E}$
is equal to the tangent distribution of $U_I^{(\alpha_0)}$ in a neighbourhood of 
$\pi^{(\alpha_0)}_0(\gamma g_0 P_I)$, we obtain \eqref{eq:trans}.
Otherwise, we apply the argument of Case~2(a) to the
distribution $D(\pi^{(\alpha_0)}_0\circ \gamma)\mathcal {E}+\mathcal{E}$ to construct another distribution.
This is possible provided that this new distribution satisfied \eqref{eq:v_g_0}.
Hence, it remains to consider the last case.

\vspace{0.5cm}

\noindent\underline{Case 2(b)}:
 $N_I^{(\alpha_0)}=\exp(\mathfrak{g}_{\alpha_0}+\mathfrak{g}_{2\alpha_0})$,
 and \begin{equation}
\label{eq:v_g_0_2}
\mathfrak{g}_{-\alpha_0}+\mathfrak{p}_I \subseteq \mathcal{V}(g_0).
\end{equation}
As in the previous cases, we show that for arbitrary $p_0\in M$,
there exists $\gamma_i\in \Gamma$  such that \eqref{eq:gamma_i} holds.
Let us suppose that for all sufficiently large $i$,
we have 
$$
D(\pi^{(\alpha_0)}_0\circ \gamma_i)_{g_0P_I}\mathcal {E}(g_0 P_I)\subset \mathcal {E}(\pi^{(\alpha_0)}_0(g_iP_I)).
$$
Then we deduce that
$$
D(\pi^{(\alpha_0)}_0\circ g_i)_{eP_I} \hbox{Ad}((u_i p_0 a_{t_i})^{-1})\mathcal {V}(g_0)\subset \mathcal {E}(\pi^{(\alpha_0)}_0(g_iP_I)),
$$
where 
\begin{align*}
\hbox{Ad}(( u_i p_0 a_{t_i})^{-1}) \mathcal {V}(g_0)
= \hbox{Ad}(p_0^{-1}) \hbox{Ad}( (u_ia_{t_i})^{-1}) \mathcal {V}(g_0).
\end{align*}
It follows from \eqref{eq:v_g_0_2} that the space $\mathcal {V}(g_0)$ is preserved by 
$\hbox{Ad}((u_i a_{t_i})^{-1})$. Hence, we conclude that
$$
D(\pi^{(\alpha_0)}_0\circ g_0)_{eP_I} \hbox{Ad}(p_0^{-1})\mathcal {V}(g_0)\subset \mathcal {E}(g_0P_I).
$$
Let $\mathcal{W}$ be the subspace generated by $\hbox{Ad}(p_0^{-1})\mathcal {V}(g_0)$, $p_0\in M$.
Then
$$
D(g_0)_{eP_I} D(\pi^{(\alpha_0)}_0)_{eP_I} \mathcal {W}\subset \mathcal {E}(g_0P_I),
$$
and 
$$
\mathcal{W}=D(\pi^{(\alpha_0)}_0)_{eP_I} \mathcal {W}\subset \mathcal {V} (g_0).
$$
Since $\mathcal{W}$ is $\hbox{Ad}(M)$-invariant, and \eqref{eq:v_g_0_2} holds,
it follows that 
$$
\mathcal {W}=\mathcal{W}'+\mathfrak{g}_{-\alpha_0}+\mathfrak{p}_I
$$
where $\mathcal{W}'$ is an $\hbox{Ad}(M)$-invariant subspace of $\mathfrak{g}_{-2\alpha_0}$.
Moreover, it follows from \eqref{eq:transverse2} that $\mathcal{W}'\ne 0$.
Hence, by Lemma \ref{l:m}, 
$\mathcal {W}=\mathfrak{g}_{-2\alpha_0}+\mathfrak{g}_{-\alpha_0}+\mathfrak{p}_I$.
However,
 we have assumed that
the distribution $\mathcal{E}$ is not equal to the tangent distribution of $U_I^{(\alpha_0)}$.
This contradiction shows that \eqref{eq:gamma2} holds.
Hence, we obtain a new distribution 
$D(\pi^{(\alpha_0)}_0\circ \gamma)\mathcal {E}$ contained in $T(U^{(\alpha_0)}_I)$ and 
defined in a neighbourhood of $\pi^{(\alpha_0)}_0(\gamma g_0 P_I)\in g_0\mathcal{O}P_I$ such that 
$D(\pi^{(\alpha_0)}_0\circ \gamma)\mathcal {E}+\mathcal{E}$ is strictly larger than
the distribution $\mathcal{E}$.
Applying this argument inductively, 
we conclude that there exist $\gamma_1=e,\gamma_2,\ldots,\gamma_r\in \Gamma$ and $x\in \pi^{(\alpha_0)}_0(g_0\mathcal{O}P_I)$ such that \eqref{eq:eee} holds.

\vspace{0.5cm}

Now we are ready to complete the proof.
The above argument shows that every open subset of $F_I$ 
contains $x$ such that for some $\gamma_1,\gamma_2,\ldots, \gamma_r\in\Gamma$,
the foliations $\mathcal{F}_{\gamma_1},\ldots,\mathcal{F}_{\gamma_r}$ 
satisfy \eqref{eq:trans} at $x$. As we already observed, $\phi$ is 
$C^\infty$ along these foliations.
This implies that $\phi$ is smooth on a neighbourhood of $x$ in $U_I^{(\alpha_0)}$.
This shows that $\phi$, restricted to $U_I^{(\alpha_0)}$, is $C^\infty$ on an open dense set
and completes the proof of the proposition.
\end{proof}

It remains to establish the following lemma which was used in the above proof.

\begin{lemma}
	\label{l:m}
	Let $\alpha_0\in\Delta$ be a simple root. Then
	\begin{enumerate}
		\item[(i)] the representation of $\hbox{\rm Ad}(M)$ on the root space $\mathfrak{g}_{-\alpha_0}$ is irreducible.
		
		\item[(ii)] the representation of $\hbox{\rm Ad}(M)$ on the root space $\mathfrak{g}_{-2\alpha_0}$
		is irreducible provided that $\mathfrak{g}$ has no real rank-one factors.
	\end{enumerate}
\end{lemma}

\begin{proof}
We consider the complexification $\hat{\mathfrak{g}}=\mathfrak{g}\otimes \mathbb{C}$ of the Lie algebra $\mathfrak{g}$. 
We refer to \cite[Ch.~5, \S4]{onishchik} for basic facts about relationship between
real semisimple Lie algebras and their complexifications.
Let $\mathfrak{h}^+$ be a maximal commutative subalgebra of 
$$
\mathfrak{m}=\hbox{Lie}(M)=\hbox{Lie}(K)\cap \mathfrak{g}_0.
$$
Then $\mathfrak{h}=\mathfrak{h}^+\oplus \mathfrak{a}$ is a maximal commutative subalgebra of $\mathfrak{g}$, $\hat{\mathfrak{h}}=\mathfrak{h}\otimes \mathbb{C}$
is a Cartan subalgebra in $\hat{\mathfrak{g}}$,
and $\hat{\mathfrak{h}}^+=\mathfrak{h}^+\otimes \mathbb{C}$
is a Cartan subalgebra in $\hat{\mathfrak{m}}$.
Under the natural identification $\mathfrak{a}^*\simeq \hat{\mathfrak{a}}(\R)^*$,
the root system $\Phi$ of $\mathfrak{a}$ is identified with the root system of $\hat{\mathfrak{a}}$ in 
$\hat{\mathfrak{g}}$. We denote by $\hat{\Phi}$ the root system 
$\hat{\mathfrak{h}}$ on $\hat{\mathfrak{g}}$. Then
$$
\Phi\cup\{0\}=\{\beta|_{\mathfrak{a}}:\, \beta\in\hat{\Phi} \}.
$$
Moreover, there exists a choice of simple roots $\hat{\Delta}\subset \hat{\Phi}$ and 
$\Delta\subset \Phi$ such that
$$
\Delta\cup\{0\}=\{\beta|_{\mathfrak{a}}:\, \beta\in\hat{\Delta} \}.
$$
We set $\hat{\Phi}_0=\{\beta\in\hat{\Phi}:\, \beta|_{\mathfrak{a}}=0 \}$ and $\hat{\Delta}_0=\Phi\cap \hat{\Phi}_0$.
Then
$$
\hat{\mathfrak{m}}=\hat{\mathfrak{h}}\oplus\left(\bigoplus_{\beta\in\hat{\Phi}_0}  \hat{\mathfrak{g}}_\beta\right),
$$
and for $\alpha\in \Phi$,
\begin{equation}\label{eq:sum}
\mathfrak{g}_\alpha\otimes \mathbb{C}=\bigoplus_{\beta\in \hat{\Phi}: \beta|_{\mathfrak{a}}=\alpha} \hat{\mathfrak{g}}_\beta.
\end{equation}
In particular, $\Phi_0$ is the root system of the semisimple Lie algebra $\hat{\mathfrak{m}}'$
with respect to $\hat{\mathfrak{h}}\cap \hat{\mathfrak{m}}'$.
Moreover, $\hat{\Delta}_0$ provides the set of simple roots for $\hat{\Phi}_0$.

Now in \eqref{eq:sum}, let us consider the case when $\alpha=-\alpha_0$ with $\alpha_0\in \Delta$.
Let $\beta_0 \in \hat{\Delta}$ be such that $\beta_0|_{\mathfrak{a}}=\alpha_0$.
We recall (see \cite[Ch.~5, \S4.3]{onishchik}) that  there exists an involution $\omega:\hat{\Delta}\backslash \hat{\Delta}_0\to \hat{\Delta}\backslash \hat{\Delta}_0$
such that if $\beta|_{\mathfrak{a}}=\alpha_0$ for some $\beta \in \hat{\Delta}$, then either $\beta=\beta_0$ or $\beta=\beta_0^\omega$, and 
for every $\beta\in \hat{\Delta}\backslash \hat{\Delta}_0$, we have a relation
$$
\beta\circ \theta =-\beta^\omega-\sum_{\sigma\in \hat{\Delta}_0} c_{\beta,\sigma} \sigma
$$
for some $c_{\beta,\sigma}\ge 0$.
Since $\beta|_{\mathfrak{h}^+}$ is purely imaginary and $\theta|_{\mathfrak{h}^+}=id$, it follows that
$$
\overline{\beta(x)}=\beta^\omega(x)-\sum_{\sigma\in \hat{\Delta}_0} c_{\beta,\sigma} \sigma(x),\quad x\in \mathfrak{h}^+.
$$
Hence, since this equality also holds for $x\in \mathfrak{a}$, we conclude that
\begin{equation}
\label{eq:conjugation}
\overline{\beta}=\beta^\omega-\sum_{\sigma\in \hat{\Delta}_0} c_{\beta,\sigma} \sigma.
\end{equation}
For $\beta\in \hat{\Phi}$, we set 
$$
\hat{\Phi}^-(\beta)=\left\{\rho\in \hat{\Phi}^-:\, \rho=\beta-\sum_{\sigma\in \hat{\Delta}_0} n_\sigma \sigma \hbox{ with } n_\sigma \ge 0 \right\},
$$
and 
$$
V(\beta)=\bigoplus_{\rho\in \hat{\Phi}^-(-\beta)} \hat{\mathfrak{g}}_\rho.
$$
Then if $\beta_0^\omega\ne \beta_0$,
\begin{equation}
\label{eq:sum2}
\mathfrak{g}_{-\alpha_0}\otimes \mathbb{C}=V(-\beta_0)\oplus V(-\beta_0^\omega),
\end{equation}
and if $\beta_0^\omega= \beta_0$,
\begin{equation}
\label{eq:sum3}
\mathfrak{g}_{-\alpha_0}\otimes \mathbb{C}=V(-\beta_0).
\end{equation}
It is clear that 
the action of $\hat{\mathfrak{m}}$ preserves $V(-\beta_0)$ and $V(-\beta_0^\omega)$.
We claim that these actions are irreducible. It follows from the definition of 
$\hat{\Phi}^-(-\beta_0)$ that $\lambda=-\beta_0|_{\hat{\mathfrak{h}}^+}$
is a highest weight for the representation $\hat{\mathfrak{m}}$ on $V(-\beta_0)$. 
Hence, $V(-\beta_0)$ contains an irreducible subrepresentation $W$
of $\hat{\mathfrak{m}}$ with the highest weight $\lambda$. 
Let us consider arbitrary 
\begin{equation}
\label{eq:mu}
\mu=\lambda -\sum_{\sigma\in\hat{\Delta}_0} n_\sigma \sigma\hbox{ with } n_\sigma\ge 0
\end{equation}
which is dominant (that is, $\left<\mu,\sigma\right>\ge 0$ for all $\sigma\in\hat{\Delta}_0$).
It follows from the Freudenthal Multiplicity formula (see \cite[\S25.1]{fh}) that $\mu$ appears as a weight in $W$.
Now suppose that $W'$ is another irreducible subrepresentation in $V(-\beta_0)$
with the highest weight $\mu$. Then $\mu$ is dominant and of the form \eqref{eq:mu}.
This proves that if the representation $\hat{\mathfrak{m}}$ on $V(-\beta_0)$ is not irreducible, then it 
has a weight $\mu$ of $\hat{\mathfrak{h}}^+$ which has multiplicity greater than one.
Let $\beta \in \hat{\Phi}^-(-\beta_0)$ be such that $\beta|_{\hat{\mathfrak{h}}^+}=\mu$.
Since $\beta|_{\mathfrak{a}}=-\beta_0$, this implies that $\dim(\hat{\mathfrak{g}}_\beta)>1$,
but $\dim(\hat{\mathfrak{g}}_\beta)=1$ for complex semisimple Lie algebras.
This contradiction implies that 
the action of $\hat{\mathfrak{m}}$ on $V(-\beta_0)$
is irreducible. 
The same argument implies that the action of $\hat{\mathfrak{m}}$ on $V(-\beta_0^\omega)$
is also irreducible. It follows from \eqref{eq:conjugation} that the complex conjugation
applied to \eqref{eq:sum2} maps $V(-\beta_0)$ to $V(-\beta_0^\omega)$. This implies that
$\mathfrak{g}_{-\alpha_0}\otimes \mathbb{C}$ contains no non-trivial $\mathfrak{m}$-invariant subspaces which are defined over $\R$. Hence, the action of $\mathfrak{m}$ on
$\mathfrak{g}_{-\alpha_0}$ is irreducible.

Now we consider the action of $\mathfrak{m}$ on the root space $\mathfrak{g}_{-2\alpha_0}$.
The proof obviously reduces to the case when $\mathfrak{g}$ is simple.
Looking through the classification of real simple Lie algebras $\mathfrak{g}$
(see \cite[Table~9]{onishchik}), we notice that 
the only case when the Lie algebra $\mathfrak{g}$ has higher rank and $\dim (\mathfrak{g}_{-2\alpha_0})>1$ is $\mathfrak{g}=\mathfrak{sp}(p,q)$ with $p<q$.
In this case, we check irreducibility directly.
Let
$$
S=\left(
\begin{tabular}{lll}
$0$ & $\cdots$ & $1$ \\
$\vdots$ & & $\vdots$ \\
$1$ & $\cdots$ & $0$
\end{tabular}
\right)
\quad \hbox{and} \quad
J=\left(
\begin{tabular}{lll}
0 & 0 & $S$ \\
0 & $I$ & 0\\
$S$ & 0 & 0
\end{tabular}
\right),
$$
where $S$ has dimension $p$ and $I$ is the identity matrix of dimension $q-p$. Then
$$
\mathfrak{g}=\mathfrak{sp}(p,q)=\{X\in \mathfrak{sl}(p+q,\mathbb{H}): \, X^* J+J X=0\},
$$
or more explicitly,
$$
\mathfrak{g}=\left\{\left(
\begin{tabular}{ccc}
$X_{11}$ & $X_{12}$ & $X_{13}$ \\
$X_{21}$ & $X_{22}$ & $-X_{12}^*J$\\
$X_{31}$ & $-JX_{21}^*$ & $-JX_{11}^*J$
\end{tabular}
\right):\, \begin{tabular}{ll}
$X_{22}^*=-X_{22}$, & $\hbox{tr}(X_{22})=0$,\\
$X_{31}^*=-JX_{31}J$, &$X_{13}^*=-JX_{13} J$
\end{tabular} 
\right\}.
$$
Its Cartan subalgebra is 
$$
\mathfrak{a}=\{\hbox{diag}(a_1,\ldots,a_p,0,\ldots,0, -a_p,\ldots, -a_1):\, a_1,\ldots,a_p\in\R  \}.
$$
The roots of the form $-2\alpha_0$ with simple $\alpha_0$ 
are given by $-2a_1,\ldots,-2a_p$, and the corresponding root spaces are
$$
\mathfrak{g}_{-2a_i}=\{\hbox{diag}(0,\ldots,0,x_i,0,\ldots,0)S:\,\; x_i\in \mathbb{H},\, x_i^*=-x_i \}.
$$
We observe that $M$ contains a copy of $\hbox{SU}(2)^p$, and one of the $\hbox{SU}(2)$-factors acts on $\mathfrak{g}_{-2a_i}$ as 
$x_i\mapsto gx_ig^*$, $g\in \hbox{SU}(2)$. This representation is irreducible.
\end{proof}

\section{Completion of the proof}
\label{sec:final}

In this section we complete the proof of our main theorem (Theorem \ref{th:main}) using results established in the previous section. We note that it follows from \cite{Dani} 
that given a $C^0$-semi-conjugacy $\psi:F=G/P\to M$, there exist
a $G$-equivariant factor map $\pi:F\to F_I$, where $F_I=G/P_I$ is another flag manifold with $P_I\supset P$,
and a $C^0$-conjugacy $\phi:F_I\to M$ such that $\psi=\phi\circ \pi$.
Hence, it remains to prove that the homeomorphism $\phi$ is smooth.
The idea for the proof is that 
the conjugacy map $\phi:F_I\to M$ is smooth along 
open dense subsets of the submanifold $U_I^{(\alpha _0)}$ (see Proposition \ref{smoothness along subfoliations}) and its translates by $\rho_0(\delta)$, $\delta\in\Gamma$,
and moreover we have smooth projections to these submanifolds in both $F_I$ and $M$ 
(see Proposition \ref{equivariant projection}) that  determine a point in some open set uniquely and smoothly --- as strings of a \emph{marionettes} puppet. 
We note while the higher rank assumption on $G$ absolutely crucial in Section \ref{s:foliations}, marionettes's argument presented here is applicable to general Zariski dense subgroups.

\begin{proof}[Proof of Theorem \ref{th:main}]
For $\alpha\in\Delta\backslash I$ and $\delta \in \Gamma$, we define 
$$
\pi^{(\alpha)} _{\delta} =  \pi^{(\alpha)} _0 \rho _0 (\delta ).
$$
Each of these maps is defined on the open dense set $\rho_0(\delta)^{-1}U_I$ of $F_I$.
By Proposition \ref{smoothness along subfoliations}, the map $\phi$, restricted to $U_I^{(\alpha)}$, is $C^\infty$ on an open dense subset $\mathcal{U}_{\alpha}$
of $U_I^{(\alpha)}$. Since the maps $\pi_0^{(\alpha)}:U_I\to U_I^{(\alpha)}$ are open, the sets $\mathcal{V}_\alpha=(\pi_0^{(\alpha)})^{-1}(\mathcal{U}_{\alpha})$ are open
dense subset of $U_I$ and, hence, of $F_I$.
By the Baire category theorem, the intersection of the sets $\delta^{-1}(\mathcal{V}_\alpha)$, $\delta\in \Gamma$ and $\alpha\in\Delta\backslash I$,
is non-empty. Let us pick a point $x_0$ that belongs to this intersection. 

Let $\mathcal{K}^{(\alpha)}$ be the distribution on $U_I$ defined by $\hbox{ker}(D(\pi^{(\alpha)}_0)_x)$, $x\in U_I$.
We claim that 
\begin{equation}
\label{eq:inter}
\bigcap _{\alpha\in\Delta\backslash I,\, \delta\in\Gamma}  D(\delta)^{-1}_{x_0}\, \mathcal{K}^{(\alpha)}(\delta \cdot x_0) = 0.
\end{equation}
Suppose that \eqref{eq:inter} fails. Then there exists non-zero vector $v\in T_{x_0}(F_I)$ such that 
$$
D(\delta)_{x_0} v \in \mathcal{K}^{(\alpha)}(\delta\cdot x_0)\quad \hbox{for all $\alpha\in \Delta\backslash I$ and $\delta\in \Gamma$.}
$$
We observe that the projection maps $\pi^{(\alpha)}_0$ are algebraic with respect to the Lie coordinates on $U_I$, so that
the distributions $\mathcal{K}^{(\alpha)}$ are also algebraic, and 
it follows from Zariski density of $\Gamma$ that 
$$
D(g)_{x_0} v \in \mathcal{K}^{(\alpha)}(g\cdot x_0)\quad \hbox{for all $g\in G$ such that $g\cdot x_0\in U_I$.}
$$
Without loss of generality, we may assume that $x_0=eP_I$. 
We recall from Lemma \ref{l:proj} that $\pi_0^{(\alpha)}$ is realised 
as a limit of the maps $\delta_n^{(\alpha)}$ on $U_I$. Hence, it follows that for every $g\in N_I^- P_I$, 
\begin{equation}\label{eq:con0}
D(\delta_n^{(\alpha)} g)_{eP} v\to 0.
\end{equation}
We write $g=up$ with $u\in N_I^-$ and $p\in P_I$.
Then $\delta_n^{(\alpha)}u(\delta_n^{(\alpha)})^{-1}$ converges in $G$.
Hence, \eqref{eq:con0} is equivalent to 
\begin{equation}\label{eq:con}
D(\delta_n^{(\alpha)} p)_{eP} v\to 0
\end{equation} 
for all $p\in P_I$.
We make the identification $T_{eP}(F_I)\simeq \mathfrak{g}/\mathfrak{p}_I$ so that 
$D(p)_{eP}w=\hbox{Ad}(p)w$ for $p\in P_I$. Then $\left<\hbox{Ad}(P_I)v\right>$
gives a non-zero $\hbox{Ad}(P_I)$-invariant subspace of $\mathfrak{g}/\mathfrak{p}_I$
such that 
$$
\hbox{Ad}(\delta_n^{(\alpha)})w \to 0\quad\hbox{for all $w\in \left<\hbox{Ad}(P_I)v\right>$}.
$$
This is equivalent to $\left<\hbox{Ad}(P_I)v\right>$ being
contained in $\mathfrak{u}^{(\alpha)}_I+\mathfrak{p}_I\subset \mathfrak{g}/\mathfrak{p}_I$ where $\mathfrak{u}^{(\alpha)}_I$ denotes
the subspace of $\mathfrak{n}^{-}_I$ spanned by all root spaces with roots not proportional to $\alpha$. Since this property must hold for all $\alpha\in\Delta\backslash I$, we obtain a contradiction with Lemma \ref{l:gp} below. 
Hence, we conclude that \eqref{eq:inter} holds.
Then there exist $(\alpha_1,\delta_1),\ldots, (\alpha_\ell,\delta_\ell)\in (\Delta\backslash I)\times \Gamma$ such that 
$$
\bigcap_{i=1}^\ell D(\delta_i )^{-1}_{x}\, \mathcal{K}^{(\alpha_i)}(\delta_i \cdot x) = 0.
$$
This property still holds in an open neighbourhood $\mathcal{O}$ of $x$ in $F_I$. 
Hence, we conclude that the map 
$$ \Pi = \left( \pi^{(\alpha_1)}_{\delta _1},\ldots, \pi^{(\alpha_\ell)}_{\delta _\ell} \right) : \mathcal{O}  \to \prod_{i=1}^\ell U_I^{(\alpha _i)} $$
is an immersion, and hence a local diffeomorphism onto its image.  
Similarly, we also get maps 
$$
\tilde{\pi}^{(\alpha_i)} _{\delta_i}=\tilde{\pi}^{(\alpha_i)}_0\rho(\delta_i):\phi (\mathcal{O})\to \phi (U^{(\alpha_i)}_I)
$$
and define
$$ 
\tilde \Pi = \left( \tilde \pi^{(\alpha_1)}_{\delta _1},\ldots, \tilde \pi^{(\alpha_\ell)}_{\delta _\ell} \right) : \phi(\mathcal{O})  \to \prod_{i=1}^\ell \phi(U_I^{(\alpha _i)} ).
$$
We recall that by Lemma \ref{l:manifold}, $\phi(U_I^{(\alpha _i)})$'s are immersed submanifolds of $U_I$.
Let us consider the following commutative diagram:
\begin{diagram}
	\mathcal{O} &\subset F_I  \rTo^\phi   &\phi(\mathcal{O})\subset M \\
	\dTo^\Pi  & &\dTo_{\tilde{\Pi}} \\ 
	\prod_{i=1}^\ell U_I^{(\alpha _i)} & \rTo^{\Phi} & \prod_{i=1}^\ell \phi(U_I^{(\alpha _i)} )
\end{diagram}
where $\Phi=\prod_{i=1}^\ell \phi|_{U_I^{(\alpha _i)}}$.
We observe that if the neighbourhood $\mathcal{O}$ is sufficiently small,
$$
\Pi(\mathcal{O})\subset \prod_{i=1}^\ell \mathcal{U}_{\alpha_i}.
$$
Since $\phi|_{\mathcal{U}_{\alpha_i}}$ is smooth, and $\phi$ is a homeomorphism,
it follows that $\phi|_{U_I^{(\alpha_i)}}$ is a local diffeomorphism on an open dense set.
We conclude that 
$$
\phi ^{-1}  = \Pi ^{-1} \Phi ^{-1} \tilde{\Pi}
$$
on a non-empty open subset of $\phi(\mathcal{O})$.
This implies that $\phi$ is a local diffeomorphism on a non-empty open subset $\mathcal{U}$ of $F_I$.
We take arbitrary $x \in F_I$.  Since $\Gamma$ acts minimally on $F_I$ (see \cite{Dani1}), there is $\gamma \in \Gamma$ such that $\rho_0(\gamma) x \in \mathcal{U}$.   By equivariance, $\phi$ is $C^{\infty}$ on $\rho_0(\gamma) ^{-1} (\mathcal{U})$,  which is an open neighbourhood of $x$. 
Furthermore, the rank of $\phi$ is maximal everywhere by minimality. Since $\phi$ is a homomorphism, this implies that $\phi$ is a diffeomorphism.
\end{proof}

It remains to prove the following lemma which was used in the above proof.

\begin{lemma}	\label{l:gp}
	Let $V$ be an $\hbox{\rm Ad}(P_I)$-invariant subspace of $\mathfrak{g}$
	which properly contains $\mathfrak{p}_I$. Then $V\cap \mathfrak{g}_\alpha\ne 0$ for some $\alpha\in \Delta\backslash I$. 
	In particular, any 
	non-trivial $\hbox{\rm Ad}(P_I)$-invariant subspace of $\mathfrak{g}/\mathfrak{p}_I$ 
	is not contained 
	in the intersection of $ \mathfrak{u}^{(\alpha)}_I+\mathfrak{p}_I$, $\alpha\in\Delta\backslash I$.
\end{lemma}

\begin{proof}
	We first consider the case when the Lie algebra $\mathfrak{g}$ is complex semisimple.
	Then the root spaces $\mathfrak{g}_\alpha$, $\alpha\in\Phi$, are one-dimensional and $[\mathfrak{g}_{\alpha_1},\mathfrak{g}_{\alpha_2}]=\mathfrak{g}_{\alpha_1+\alpha_2}$
	for all $\alpha_1,\alpha_2\in \Phi$ such that $\alpha_1+\alpha_2\ne 0$
	(see \cite[Ch.III, \S4]{Helgason}).
	%We recall that $\mathfrak{p}_I=\mathfrak{s}_I\oplus \mathfrak{n}_I^+$
	%where $\mathfrak{s}_I=\mathfrak{m}_I\oplus \mathfrak{a}_I$ is a reductive subalgebra
	%of $\mathfrak{p}_I$ which contains a Cartan subalgebra.
	%Then $V=V'\oplus \mathfrak{p}_I$ where $V'$ is a non-trivial $\mathfrak{s}_I$-invariant
	%subspace of $V$. 
	Since $V$ is $\hbox{Ad}(P_I)$-invariant,
	it is also invariant under the action of the Cartan subalgebra contained in $\mathfrak{p}_I$, and it follows that  $V$ is equal to a sum of root spaces $\mathfrak{g}_{\rho}$ for some $\rho \in \Phi$.
	
	We claim that $\mathfrak{g}_{-\beta}\subset V$
	for some $\beta\in\Delta\backslash I$. To prove this claim, we take $\beta\in \Phi^+\backslash \Phi_I$ such that $\mathfrak{g}_{-\beta}\subset V$, and $\beta$
	is `minimal' in the sense that $\mathfrak{g}_{-(\beta-\alpha)}\nsubseteq V$ for every $\alpha\in \Delta$ such that $\beta-\alpha\in\Phi^+\backslash \Phi_I$. Suppose that $\beta\notin \Delta$.  
	Since $V$ is $\hbox{Ad}(P_I)$-invariant, it also follows that for every $\alpha\in \Delta$,
	$$
	[\mathfrak{g}_\alpha,\mathfrak{g}_{-\beta}]=\mathfrak{g}_{-(\beta-\alpha)}
	\subset V.
	$$
	Hence, according to our choice of $\beta$, either 
	$\mathfrak{g}_{-(\beta-\alpha)}=0$, i.e., $\beta-\alpha\notin \Phi$,
	or $\beta-\alpha\in \Phi_I$.
	We have
	$$
	\beta\in \sum_{\alpha\in \Delta\backslash I} n_\alpha \alpha +\left<I\right>
	$$
	for some $n_\alpha\ge 0$ with $\sum_{\alpha\in \Delta\backslash I} n_\alpha>0$.
	First, we consider the case when $\sum_{\alpha\in \Delta\backslash I} n_\alpha\ge 2$.
	Then clearly, $\beta-\alpha\notin \Phi_I$ for all $\alpha\in\Delta$,
	and it follows that $[\mathfrak{g}_\alpha,\mathfrak{g}_{-\beta}]=0$ for all simple roots $\alpha$.
	This means that $\mathfrak{g}_{-\beta}$ consists of highest weight vectors
	for the adjoint representation of $\mathfrak{g}$.
	In particular, it follows that $-\beta$ must be dominant,
	 but this is impossible because $\beta\in \Phi^+$
	 (recall that all dominant roots are contained in $\Phi^+$). Hence, we conclude that
	$$
	\beta\in \alpha_0 +\left<I\right>
	$$
	for some $\alpha_0\in \Delta\backslash I$. Then for all $\alpha\in\Delta\backslash \{\alpha_0\}$, $\beta-\alpha\notin \Phi_I$, and 
	it follows as before that $[\mathfrak{g}_\alpha,\mathfrak{g}_{-\beta}]=0$ when $\alpha\in\Delta\backslash \{\alpha_0\}$.
	If we also had that  $[\mathfrak{g}_{\alpha_0},\mathfrak{g}_{-\beta}]=0$,
	then $\mathfrak{g}_{-\beta}$ would have consisted of the highest weight
	vectors of for the adjoint representation.
	Then we would get a contradiction as before.
	Therefore, since we assumed that $\beta\notin \Delta$, we conclude that 
	$$
	[\mathfrak{g}_{\alpha_0},\mathfrak{g}_{-\beta}]=\mathfrak{g}_{-(\beta-\alpha_0)}\ne 0,
	$$
	so that $\beta-\alpha_0$ is a root, and since $\beta\in \Phi^+$, 
	$\beta-\alpha_0\in \Phi^+$ too. Then since $V$ is $\hbox{Ad}(P_I)$-invariant,
	$$
	[\mathfrak{g}_{\beta-\alpha_0},\mathfrak{g}_{-\beta}]=\mathfrak{g}_{-\alpha_0}\subset V.
	$$
	This proves the lemma when $\mathfrak{g}$ is a complex semisimple Lie algebra.
	
	To treat the general case, we consider the complex semisimple Lie algebra
	$\hat{\mathfrak{g}}=\mathfrak{g}\otimes \mathbb{C}$. 
	It contains (see \cite[Ch.~5, \S4]{onishchik}) a Cartan subalgebra
	$\hat{\mathfrak{a}}$ of the form 
	$\hat{\mathfrak{a}}=(\mathfrak{h}^+\oplus \mathfrak{a})\otimes \mathbb{C}$,
	where $\mathfrak{h}^+$ is a Cartan subalgebra of $\mathfrak{m}$
	such that the root system $\hat{\Phi}$ and 
	the system of simple roots $\hat{\Delta}\subset \hat{\Phi}$ 
	associated to $\hat{\mathfrak{a}}$ satisfy
	$$
	\{\alpha|_{\mathfrak{a}}:\alpha\in \hat{\Delta}\}=\Delta\cup \{0\}, 
	$$ 
	and
	$$
	\mathfrak{g}_\beta\otimes \mathbb{C}= \bigoplus_{\rho \in \hat{\Phi}:\, \rho|_{\mathfrak{a}}=\beta}\;  \hat{\mathfrak{g}}_\rho \quad\hbox{for all $\beta\in \Phi.$}
	$$
	We set 
	$$
	\hat{I}=\{\alpha\in \hat{\Delta}:\, \alpha|_{\mathfrak{a}}\in I\cup\{0\}\}.
	$$
	Then $\mathfrak{p}_I\otimes \mathbb{C}=\hat{\mathfrak{p}}_{\hat{I}}$.
The previous discussion implies that $\hat{\mathfrak{g}}_{-\alpha}\subset V\otimes \mathbb{C}$ for some $\alpha\in \hat{\Delta}\backslash \hat{I}$, and hence $(\mathfrak{g}_{-\beta}\otimes \mathbb{C})\cap (V\otimes \mathbb{C})\ne 0$ for $\beta\in \Delta\backslash I$ such that $\beta=\alpha|_{\mathfrak{a}}$. This also implies 
	that $\mathfrak{g}_{-\beta}\cap V\ne 0$ and proves the lemma.
\end{proof}

 \section{On topological Sinks}
 \label{sec:sink}
 
 A key assumption for our main result (Theorem \ref{th:main}) is that the lattice action on the manifold $M$ has at least one differentiable sink. While existence of a topological 
 sink is clear for continuous factors of projective actions, we note that a topological sink of a lattice action is
 not always a differentiable sink. In this section, 
 we give an example of a smooth lattice action on a circle bundle over a flag manifold
 to illustrate this.
 
 Let $f_t$ be a flow on the circle $S^1\simeq \R/\Z$ such that $0\in \R/\Z$ is a topological sink for the maps $f_t$, $t> 0$, but it is not a differentiable sink.
  For example, we could take the flow for the vector field on $\R/\Z$ given in local coordinates around $0$ by $x^3 -x$, $x\in \R$.   
 Let $G$ be a non-compact connected semisimple Lie group and $\Gamma$ a lattice subgroup of $G$.
 We choose a Cartan subgroup $A$ of $G$ such that $A\cap \Gamma$ is a lattice in $A$.
 Such a subgroup exists by \cite{Prasad-Raghunathan}. We fix a set $\Delta$ of simple roots 
 and a root $\alpha\in \Delta$. This determines a minimal parabolic subgroup $P$ containing $A$.
 The root $\alpha$ gives a homomorphism $\log(\alpha):P\to \R$.
 Since $A\cap \Gamma$ is a lattice in $A$, there exists $\gamma_0\in A\cap\Gamma$ such that
 $\alpha(\gamma_0)>1$ for all $\alpha\in\Delta$. Then $eP$ is a smooth sink for the action of 
 $\gamma$ on $G/P$. We consider the equivalence relation on $G\times S^1$ defined by
 $$
 (g,x)\sim (gp^{-1}, f_{\log(\alpha)(p)}(x)),\quad g\in G,\; p\in P,\; x\in S^1.
 $$
 Then 
 $$
 M=(G\times S^1)/\sim
 $$
 is a manifold equipped with the smooth $\Gamma$-action
 $$
 [g,x]\mapsto [\gamma g,x],\quad [g,x]\in M,\; \gamma\in\Gamma
 $$
 It is clear $[e,0]\in M$ is fixed by $\gamma_0$. Moreover, for $v\in \mathfrak{n}^-$ and $x\in S^1$,
 $$
 \gamma_0\cdot [\exp(v),x]=\left[\exp(\hbox{Ad}(\gamma_0)v), \phi_{\log(\alpha)(\gamma_0)}(x)\right].
 $$
 Hence, $[e,0]$ is a topological sink of $\gamma_0$, but it is not a differentiable sink.

% \bibliography{projective.bib}{}

\begin{thebibliography}{100}
 
\bibitem{Benoist} 
Y. Benoist, Propri\'et\'es asymptotiques des groupes lin\'eaires. Geom. Funct. Anal. 7 (1997), no. 1, 1--47.

\bibitem{Borel}
A. Borel, Linear algebraic groups. Graduate Texts in Mathematics, 126. Springer-Verlag, New York, 1991.

\bibitem{BRHW2016}
A. Brown, F. Rodriguez Hertz and Z. Wang,  Invariant measures and measurable projective factors for actions of higher-rank lattices on manifolds,
	arXiv:1609.05565.
 
\bibitem{Dani} S. G. Dani, Continuous equivariant images of lattice-actions on boundaries. Ann. of Math. (2) 119 (1984), no. 1, 111--119.

\bibitem{Dani1} S. G. Dani, Orbits of horospherical flows. Duke Math. J. 53 (1986), no. 1, 177--188.

\bibitem{Fisher} D. Fisher, Groups acting on manifolds: around the Zimmer program. Geometry, rigidity, and group actions, 72--157, Chicago Lectures in Math., Univ. Chicago Press, Chicago, IL, 2011.

 
\bibitem{Fisher-Kalinin-Spatzier}
D. Fisher, B. Kalinin, R. Spatzier, 
Global rigidity of higher rank Anosov actions on tori and nilmanifolds. With an appendix by James F. Davis. J. Amer. Math. Soc. 26 (2013), no. 1, 167--198.

\bibitem{Fisher-Margulis} 
D. Fisher and G. Margulis, 
Local rigidity of affine actions of higher rank groups and lattices.
Ann. of Math. (2) 170 (2009), no. 1, 67--122.

\bibitem{fh} W. Fulton and J. Harris, Representation theory. A first course. Graduate Texts in Mathematics, 129. Readings in Mathematics. Springer-Verlag, New York, 1991. 
 
\bibitem{Ghys} 
\'E. Ghys, Rigidit\'e diff\'erentiable des groupes fuchsiens. Inst. Hautes \'Etudes Sci. Publ. Math. No. 78 (1993), 163--185.

\bibitem{G} M. Guysinsky,
The theory of non-stationary normal forms.
Ergodic Theory Dynam. Systems 22 (2002), no. 3, 845--862.

\bibitem{GK} 
M. Guysinsky and A. Katok, 
Normal forms and invariant geometric structures for dynamical systems with invariant contracting foliations. 
Math. Res. Lett. 5 (1998), no. 1--2, 149--163. 

\bibitem{Helgason}
S. Helgason, Differential geometry, Lie groups, and symmetric spaces.  Graduate Studies in Mathematics, 34. American Mathematical Society, Providence, RI, 2001.

\bibitem{Ivanov} 
N. V. Ivanov, Action of M\"obius transformations on homeomorphisms: stability and rigidity. Geom. Funct. Anal. 6 (1996), no. 1, 79--119. 

\bibitem{Kanai} 
M. Kanai, A new approach to the rigidity of discrete group actions. Geom. Funct. Anal. 6 (1996), no. 6, 943--1056.
 
\bibitem{Katok-Spatzier}  
A. Katok and R. Spatzier, Differential rigidity of Anosov actions of higher rank abelian groups and algebraic lattice actions. Tr. Mat. Inst. Steklova 216 (1997), Din. Sist. i Smezhnye Vopr., 292--319; reprinted in Proc. Steklov Inst. Math. 1997, no. 1 (216), 287--314.

\bibitem{Kleinbock-Margulis} 
D. Y. Kleinbock and G. A. Margulis, Bounded orbits of nonquasiunipotent flows on homogeneous spaces. Sinai's Moscow Seminar on Dynamical Systems, 141--172, Amer. Math. Soc. Transl. Ser. 2, 171, Amer. Math. Soc., Providence, RI, 1996.

\bibitem{margulis} G. A. Margulis, Finiteness of quotient groups of discrete subgroups.  Funktsional. Anal. i Prilozhen. 13 (1979), no. 3, 28--39.

\bibitem{margulis-book} 
G. A. Margulis, Discrete subgroups of semisimple Lie groups. Ergebnisse der Mathematik und ihrer Grenzgebiete (3), 17. Springer-Verlag, Berlin, 1991.

\bibitem{onishchik} A. L. Onishchik and E. B. Vinberg, Lie groups and algebraic groups. Translated from the Russian and with a preface by D. A. Leites. Springer Series in Soviet Mathematics. Springer-Verlag, Berlin, 1990.
  
%\bibitem{Margulis-Soifer}  Margulis, G. A.; Soefer, G. A. Maximal subgroups of infinite index in finitely generated linear groups. J. Algebra 69 (1981), no. 1, 1?23. 

\bibitem{Prasad-Raghunathan}
G. Prasad and M. S. Raghunathan, Cartan subgroups and lattices in semi-simple groups. Ann. of Math. (2) 96 (1972), 296--317.

\bibitem{Prasad-Rapinchuk}
G. Prasad and A. Rapinchuk, Existence of irreducible R-regular elements in Zariski-dense subgroups. Math. Res. Lett. 10 (2003), no. 1, 21--32.

\bibitem{RH-Wang}  
F. Rodriguez Hertz and Z. Wang, Global rigidity of higher rank Anosov algebraic actions, 
arXiv:1304.1234; to appear in Invent. Math.

\bibitem{spatzier}  
R. J. Spatzier, On lattices acting on boundaries of semisimple groups. 
Ergodic Theory Dynamical Systems 1 (1981), no. 4, 489--494.

\bibitem{Tukia} 
P. Tukia, Differentiability and rigidity of Möbius groups. Invent. Math. 82 (1985), no. 3, 557--578. 

\bibitem{zimmer0} R. Zimmer, 
Equivariant images of projective space under the action of SL(n,Z). 
Ergodic Theory Dynamical Systems 1 (1981), no. 4, 519--522. 

\bibitem{zimmer}
R. J. Zimmer, Ergodic theory and semisimple groups. Monographs in Mathematics, 81. Birkh\"auser Verlag, Basel, 1984.

\end{thebibliography}
%\bibliographystyle{alpha}

%\end{document}
%\end
%%%%%%%%%%%%%%%%%%%%%%%%%%%%%%%%%%%%%%%%%%%%%%%%%%%%%%%%%%%%%%%%%%%%%%%%%%%%%%%%%%%%%%%%%%%%%%%%%%%%%%%%%%%%%%%
%%%%%%%%%%%%%%%%%%%%%%%%%%%%%%%%%%%%%%%%%%%%%%%%%%%%%%%%%%%%%%%%%%%%%%%%%%%%%%%%%%%%%%%%%%%%%%%%%%%%%%%%%%%%%%%

\end{document}